\documentclass[11.5pt]{article}

\usepackage{amsmath,amsfonts}
\usepackage{caption}
\usepackage{accents}
\usepackage{multirow}

\usepackage{lineno,hyperref}
\modulolinenumbers[5]

\usepackage[T1]{fontenc}
\usepackage{color}
\usepackage{fancyhdr}
\usepackage{fullpage}
\usepackage{setspace}
\doublespace
\usepackage{comment}
\usepackage{makeidx}
\usepackage{hyperref}
\usepackage{cleveref}
\crefformat{footnote}{#2\footnotemark[#1]#3}

\usepackage[toc,page]{appendix}
\usepackage{multicol}
\usepackage{amsmath, amsfonts, amssymb}
\usepackage{amsthm}
\usepackage{graphicx}
\usepackage{pst-all,pst-infixplot,pst-math,pst-fractal,pst-solides3d}

\newtheorem{teo}{Theorem}[section]
\newtheorem{prop}{Proposition}[section]

\newtheorem{lemma}{Lemma}[section]
\newtheorem{corollary}{Corollary}[section]
\newtheorem{assumption}{Assumption}[section]

\theoremstyle{definition}
\newtheorem{defi}[teo]{Definition}

\theoremstyle{remark}
\newtheorem{remark}{Remark}[section]

\begin{document}

\author{Manuel Guerra\thanks{CEMAPRE, Instituto Superior de Economia e Gest\~ao, Universidade de Lisboa, Rua do Quelhas 6, Lisbon, Portugal, Email: mguerra@iseg.ulisboa.pt}, \
Cl\'audia Nunes\thanks{Department of Mathematics and
CEMAT, Instituto Superior T\'ecnico, Universidade de
Lisboa, Av. Rovisco Pais, 1049-001 Lisboa, Portugal,
Email: cnunes@math.tecnico.ulisboa.pt}
\ and Carlos Oliveira\thanks{Department of Mathematics and
CEMAT, Instituto Superior T\'ecnico, Universidade de
Lisboa, Av. Rovisco Pais, 1049-001 Lisboa, Portugal,
Email: carlosmoliveira@tecnico.ulisboa.pt}}

\title{\sc On a Class of Optimal Stopping Problems with Applications to Real Option Theory}

\maketitle

\begin{abstract}

We consider an optimal stopping time problem  related with many models found in real options problems. The main goal of this work is to bring for the field of real options, different and more realistic pay-off functions, and negative interest rates. Thus, we present analytical solutions for a wide class of pay-off functions, considering quite general assumptions over the model. Also, an extensive and general sensitivity analysis to the solutions, and an economic example which highlight the mathematical difficulties in the standard approaches, are provided.

\mbox{}
\newline
{\bf Keywords and phrases.\/} Control, optimal stopping times, real options, free-boundary problems.


\mbox{}
\newline
{\bf AMS (2010) Subject Classifications.\/} Primary 93E20;
secondary 60G40, 90B50.
\end{abstract}

%
%

\section{Introduction}
Optimal stopping problems have been studied intensively in the mathematical context. Particularly, significant contributions have been made in the past decades motivated by financial applications. We refer to Peskir and Shiryaev \cite{Shiryaev} for a recent and important survey on optimal stopping and free-boundary problems, and the methods commonly used to solve such problems. 

In this work, we address some questions related to optimal stopping problems, with particular focus on possible applications to Real Option Theory. In the real option framework, there is the right, but not the obligation, to undertake certain business initiatives, such as deferring, abandoning, expanding, staging or contracting a capital investment project, according to the uncertainty of the market and the partial or total irreversibility of the decisions. Moreover, the owner of his option chooses the right moment to embrace the decision, in order to maximize the value of the associated project. We refer to Dixit and Pindyck \cite{DixitPindyck}, McDonald and Siegel \cite{McDonald:Siegel:86}, Dixit \cite{dixit1989entry}, Trigeorgis \cite{trigeorgis} and references therein for good examples of seminal works on real options.

We consider the stochastic process, hereby denoted by $X=\{X(t), t \geq 0\}$, which follows a geometric Brownian motion (GBM) given by the stochastic differential equation:
\begin{equation}\label{GBM}
dX(t)=\alpha X(t)dt+\sigma X(t)dW(t), \quad X(0)=0,
\end{equation}
where $W=\{W(t), t \geq 0\}$ represents the (standard) Brownian motion and $\alpha,\sigma\in\mathbb{R}$. Following the classical references of real options, as the ones above mentioned, the process $X$ represents a stochastic economic indicator, for example, the price of or the demand for a product.  Furthermore, we let $\Pi$ be the running function  (i.e., the function that quantifies the revenue associated with the economic indicator). Since the running function $\Pi$ may have different shapes, discontinuities and behaviours in different problems, in this paper, we will assume general conditions over $\Pi$, that will be stated later on. 

We introduce the following functional:
\begin{align}
J(x,\tau)=E\left[\int_0^{\tau}e^{-rs}\Pi(X(s))ds\vert X(0)=x\right],  \label{optimal1}
\end{align}
where $\tau$ is a stopping time adapted to the filtration generated by the stochastic process $X$. Finally, the  problem  that we want to solve can be defined as follows:
\begin{equation}
V(x)=\sup_{\tau\in S}J(x,\tau),\label{optimal00}
\end{equation}
where $S$ is the set of all stopping times which will be properly defined in the next section. Consequently, solving the optimal stopping problem \eqref{optimal00} is equivalent to finding the stopping  strategy $\tau^*\in{\cal S}$ which maximizes $J(x,.)$ for all $x\in ]0,\infty[$. For the rest of the paper, we call $V$ the \textit{value function}. Also, we will write $E_x[.]$ instead of $E[.\vert X(t)=x]$ to ease the presentation. 

Motivated by the applications that come from real options, we assume an infinite time horizon and a discounted version of stopping time problems, where the discount factor is the interest rate  (fixed and known,  hereby denoted by $r$).  
Usually, $r$ is assumed to be positive, as in the case of Peskir and Shiryaev \cite{Shiryaev} (see section 6.3) or Knudsen, Meister and Zervos \cite{knudsen1998valuation} or Guerra, Nunes and Oliveira \cite{nossoartigo}. It may also be assumed to be equal to zero, as in the case of R\"{u}schendorf and Urusov \cite{ruschendorf2008class}. Here, we relax this assumption, letting $r$ take positive and negative values.

The optimal stopping time problem defined above can be easily linked to a particular real option: the abandonment option, which is studied in different contexts (see, for instance, Dixit and Pindyck \cite{DixitPindyck}, Guerra, Nunes and Oliveira \cite{nossoartigo}, Hagspiel, Huisman, Kort and Nunes \cite{hagspiel2016escape}, Alverez \cite{alvarez1999optimal} and references therein). Another strictly related problem to the one introduced above  is the "optimal entrance time problem", which appears in real options, for example, when one discusses the optimal time to invest in a different product or market (see for instance Dixit \cite{dixit1989entry} or Stokey \cite{stokey2016wait} ) . Formally, the entrance time problem is presented as:
\begin{align}\label{entry problem}
G(x)=\sup\limits_{\tau\in{\cal S}} L(\tau,x), \quad \text{where}\quad L(\tau,x)=E_x\left[\int_{\tau}^{\infty}e^{-r s}\Pi(X(s))ds\right].
\end{align}
where $X$, $\tau$ and ${\cal S}$ are as described before. Mathematically speaking, optimal stopping times and optimal entrance times share some common characteristics and they can both be studied in the same framework as free boundary problems.  Throughout the paper, we will establish some relations between the two problems here presented.

In order to characterize the solution of the optimal stopping problem \eqref{optimal00} under our assumptions, we present a verification theorem that guarantees that the solution is continuously differentiable ($C^1$) with absolutely continuous derivative ($AC$). Under different assumptions than the ones that we will assume here, similar results can be found, for instance, in Knudsen, Meister and Zervos \cite{knudsen1998valuation}, R\"{u}schendorf and Urusov \cite{ruschendorf2008class} or Guerra, Nunes and Oliveira \cite{nossoartigo}. 

Our contribution to the state of the art on this topic is precisely to obtain closed form analytic solutions to the problem \eqref{optimal00} under different conditions of the function $\Pi$ and parameters involved. In case we make the usual assumptions about $r$ and the function $\Pi$ (namely, if $r>0$ and $\Pi$ is monotonic), our solution agrees with the known solution for the optimal exit decision of a firm with a constant abandonment cost (that can be found, for example, in  Dixit and Pindike \cite{DixitPindyck}). Notwithstanding, our characterization can also be used in other problems involving more complex profit functions. 

Regardless of the fact that the majority of research in Real Option Theory considers pay-off functions to be monotonically increasing, the possibility of having different behaviours is admitted. See, for instance, the discussion about ceilling prices in Dixit and Pindyck \cite{DixitPindyck} or the gross profit functions' behaviour described in Dahan and Srinivasan \cite{dahan2011impact}.
Thus, we bring to the real option analysis the possibility of studying pay-off functions with different shapes. Therefore, we illustrate the potential benefits from this representation by using an example of a profit function that is non-monotonic.   

Furthermore, we also provide an extensive analysis of the stopping and continuation regions' behaviour when $\alpha$ and $\sigma^2$ are changing. In particular, we get surprising results when $r<0$. In the last few years, some research has been made in order to understand how the decisions under uncertainty may be influenced when $r$ is stochastic or when $r$ is the main source of uncertainty for the project (see for instance Dias and Shackleton \cite{dias2011hysteresis} and Alvarez and Koskela \cite{alvarez2006irreversible}). Although this work does not consider stochastic interest rates, it brings some additional information about the decision making when $r$ is not constant and strictly positive.

Since there are many other applitations of optimal stopping in financial mathematics, we refer to the alphabetically-ordered list of important contributions in this field: Arkin \cite{arkin2015threshold}, Belomestny,  R{\"u}schendorf, and Urusov \cite{belomestny2010optimal}, Bronstein, Hughston, Pistorius and Zervos \cite{bronstein2006discretionary}, Chevalier, Vath, Roch and Scotti \cite{chevalier2015optimal}, Dayanik
\cite{dayanik2008optimal}, Dayanik and Egami \cite{dayanik2012optimal}, Dayanik and Karatzas \cite{dayanik2003optimal}, D{\'e}camps and Villeneuve \cite{decamps2007optimal}, Johnson \cite{johnson2015solution}, Johnson and Zervos \cite{johnson2007solution}, Lamberton and Zervos \cite{lamberton2013optimal}, Villeneuve \cite{villeneuve2007threshold}.

The paper is organized as follows: in Section 2, we present the most relevant assumptions on the profit function and the model's parameters, in Section 3, we characterize the optimal stopping strategy and, in Section 4, we provide the final result, presenting the solution to the problem. In Section 5, we study the behaviour of the stopping strategy with respect to the diffusion parameters $\alpha$ and $\sigma$. We illustrate the results derived from this paper, considering particular instances of a class of profit functions $\Pi$ in section 6 and, finally, there are two appendices with technical results.

\section{Problem set up}\label{model}
In this section we present the assumptions regarding the running function $\Pi$ and the parameters of the stochastic process $X$, $\alpha$ and $\sigma$. We consider the complete probability space $(\Omega,{\cal F},P)$, equipped with the filtration $\{{\cal F}_t\}_{t\geq 0}$ generated by the Brownian motion $W=\{W(t), t \geq 0\}$. Finally, $S$ is the set of all admissible $({\cal F}_t)$-stopping times. 

We consider an infinitesimal generator associated to the GBM, $X$, hereby denoted by  ${\cal L}$, given by:
\begin{equation}\label{infinitesimal generator GBM}
{\cal L}\phi(x):=\lim_{t\downarrow 0}\frac{E_x\left[e^{-rt}\phi(X(t))-\phi(x)\right]}{t}=-r\phi(x)+\alpha x\phi'(x)+\frac{\sigma^2}{2}x^2\phi(x),
\end{equation}
where $\phi$ has enough regularity properties. For future reference, we note that considering the change of variable  $t=ln\left(\frac{x}{x^*}\right)$ and $u(t)=\phi(e^t)$, where $x^*\in\mathbb{R}$, we can re-write the infinitesimal generator as:
\begin{equation}\label{modified infinitesimal generator}
{\cal \tilde{L}}u(t)=-ru(t)+\left(\alpha-\frac{\sigma^2}{2}\right)u'(t)+\frac{\sigma^2}{2}u''(t).
\end{equation}
The corresponding characteristic polynomial is given by:
\begin{equation}\label{characteristic polynomial}
P(d)=\frac{\sigma^2}{2}d^2+\left(\alpha-\frac{\sigma^2}{2}\right)d-r,
\end{equation}
and its roots are:
\begin{align}\label{roots}
{d_1}:=\frac{\left(\frac{\sigma^2}{2}-\alpha\right)-\sqrt{\left(\frac{\sigma^2}{2}-\alpha\right)^2+2\sigma^2r}}{\sigma^2}\quad\text{and}\quad {d_2}:=\frac{\left(\frac{\sigma^2}{2}-\alpha\right)+\sqrt{\left(\frac{\sigma^2}{2}-\alpha\right)^2+2\sigma^2r}}{\sigma^2}.
\end{align}
Therefore,
\begin{equation}\label{reparametrization}
r=-\frac{\sigma^2}{2}d_1d_2\quad \text{and} \quad \alpha=\frac{\sigma^2}{2}(1-d_1-d_2).
\end{equation}
Note that we may use either $(r,\alpha,\sigma^2)$ as natural parameters of the model or $(d_1,d_2,\sigma^2)$ as, in view of \eqref{reparametrization}, they are equivalent.
The reason to use $d_1$ and $d_2$ as parameters is that some of the assumptions that we need to consider, in order to have a non-trivial optimisation problem, are expressed in terms of these quantities. Now, we introduce the assumptions over $\Pi$, which will be further considered.


%
\begin{assumption}\label{A3}
The Borel measurable function $\Pi:]0,\infty[\to\mathbb{R}$ is such that $\Pi \in L^1_{loc}(]0,\infty[)$.
\end{assumption}
It follows trivially from the definition of the optimisation problem that: if  $\Pi(x)\geq 0$ for all $x\in ]0,\infty[$, then the optimal stopping time is  $\tau=\infty$  and, therefore, $V(x)=J(x,\infty)$. Additionally, if $\Pi(x)\leq 0$ for all $x\in ]0,\infty[$, then the optimal stopping time is zero and, thus, $V(x)=J(x,0)=0$. In these cases, we say that the optimisation problem is trivial. Also note that:
\begin{itemize}
\item  if $\tau$ is such that $E_x\left[\int_0^\tau e^{-rs}\Pi(X(s))\right]>0$ for all $x$, then
the optimal stopping time is larger or equal to $\tau$ w.p.1. (due to the continuity of the integral);
\item  if there exists $\tau$ such that $E_x\left[\int_0^\tau e^{-rs}\Pi(X(s))\right]=\infty$, then the optimal stopping time is $\tau=\infty$, and the problem is trivial.
\end{itemize}
Therefore, in addition to assumption \ref{A3}, we assume the following:
\begin{assumption}\label{A2}
$v_p^+(x):=E_x\left[\int_0^\infty e^{-rs}\Pi^+(X(s))ds\right]<\infty$\footnote{We use the following notation:  given a function $f$ we denote by $f^+=max(f,0)$ and $f^-=\max(0,-f)$.}, for all $x\in ]0,\infty[$.
\end{assumption}
Note that assumption \ref{A2} allows to have $v_p^-(x):=E_x\left[\int_0^{\infty}e^{-rs}\Pi^-(X(s))ds\right]=\infty$, which relaxes, for example, the correspondent assumption proposed by Knudsen, Meister and Zervos \cite{knudsen1998valuation}: 
\begin{equation}\label{usual}
E_x\left[\int_0^{\infty}e^{-rs}\vert \Pi(X(s)\vert ds\right]<\infty.
\end{equation}
However, as we will see later, this
does not change the regularity of the value function $V$. In the following proposition we present conditions on the parameters that should be avoided, in order to have a well defined and a non-trivial optimisation problem. 

\begin{prop}\label{P1} 
Let $X$ be defined as in \eqref{GBM} and $g:]0,\infty[\to[0,+\infty]$ a Borel measurable function, such that $\int_0^{\infty}g(x)dx\in ]0,\infty]$. If $d_1=d_2$ or $d_1=\overline{d_2}\in\mathbb{C}\setminus\mathbb{R}$, then 
\begin{equation*}
E_x\left[\int_0^{\infty}e^{-rt}g(X(t))dt\right]=\infty.
\end{equation*}
\end{prop}
\begin{proof}[Proof of Proposition \ref{P1}]
Using Fubini's theorem, we have:
\begin{align*}
E_x\left[\int_0^{\infty}e^{-rt}g(X(t))dt\right]=&\int_0^{\infty}e^{-rt}E\left[g\left(X(t)\right)\right]dt\\
=&\int_0^{\infty}e^{\frac{\sigma^2}{2}d_1d_2t}\int_{-\infty}^{\infty}g\left(xe^{-\frac{\sigma^2}{2}\left(d_1+d_2\right)t+\sigma w}\right)\frac{e^{-\frac{w^2}{2t}}}{\sqrt{2\pi t}}dwdt,
\end{align*}
where we took into account the parametrization given by \eqref{reparametrization}. Now, using the change of variable $w=\frac{1}{\sigma}\log\frac{y}{x}+\sigma\frac{d_1+d_2}{2}t$, where we assume, without loss of generality, that $\sigma>0$, it is a matter of calculations to observe that: 
\begin{align*}
E_x\left[\int_0^{\infty}e^{-rt}g(X(t))dt\right]\geq A\int_1^{+\infty}\frac{1}{\sqrt{2\pi t}}e^{\frac{\sigma^2}{2}\left(\vert d_1\vert^2-Re(d_1)^2\right)}dt,
\end{align*}
where $A=\int_0^{\infty}\frac{g(y)}{\sigma y}e^{-\frac{1}{2\sigma^2}\log\left(\frac{y}{x}\right)^2-\frac{d_1+d_2}{2}\log\left(\frac{y}{x}\right)}dy>0$ and $Re(d_1)$ represents the real part of $d_1$, which is sufficient to get the intended result.
\end{proof}
Therefore, in view of the last proposition and assumption \ref{A2}, we assume that $d_1$ and $d_2$ are real and different. Consequently, the cases (i) $d_1,d_2\in\mathbb{C}\setminus\mathbb{R}$ and (ii) $d_1,d_2\in\mathbb{R}~ \text{with}~ d_1= d_2$ will not be analysed, since these cases lead to optimisation problems that are trivial or ill-posed.  We remark that negative discount rates (interest rates) can still be considered, as long as 
\begin{equation}\label{limite inferior par r}
r>-\frac{1}{2\sigma^2}\left(\frac{\sigma^2}{2}-\alpha\right)^2.
\end{equation}
For future reference, we make the following remark:
\begin{remark}\label{l1}
According to Knudsen, Meister and Zervos \cite{knudsen1998valuation} (proposition 4.1), for a Borel measurable function $g:]0,\infty[\to]0,\infty[$, the following condition holds:
\begin{equation}\label{condicao zervos}
 E_x\left[\int_0^{\infty}e^{-rs} g(X(s)) ds\right]<\infty,
\end{equation}
if, and only if,  
$x\to x^{-d_1-1}g(x)\in L^1(]0,w[) $
and
$x\to x^{-d_2-1}g(x)\in L^1(]w,\infty[),$
for all $w\in ]0,\infty[$ and $d_1 \neq d_2$. Moreover,  if $\liminf_{x\downarrow 0} \frac{\vert\Pi(x)\vert}{x^{d_1}}\neq 0$ or $\liminf_{x\to \infty} \frac{\vert\Pi(x)\vert}{x^{d_2}}\neq 0$, then $v_p^+(x)=\infty$, for all $x\in]0,\infty[$.
\end{remark}
Henceforward, we will consider running functions $\Pi$ with the following characteristics: 
\begin{assumption}\label{A1}
The function $\Pi$ is such that 
\begin{equation}\label{ruschendorf scheme}
\Pi(x)\begin{cases}
>0,& \quad  \text{for all }x \in ]x_{1r},x_{2l}[\\
=0,& \quad \text{for all }x\in [x_{1l},x_{1r}]\cup[x_{2l},x_{2r}]\\
<0,&\quad \text{for all }x\in ]0,x_{1l}[\cup]x_{2r},\infty[
\end{cases},
\end{equation}
 with $0\leq x_{1l}\leq x_{1r}<x_{2l}\leq x_{2r}<\infty$ or $0< x_{1l}\leq x_{1r}<x_{2l}\leq x_{2r}\leq \infty$, where we adopt the usual convention $]a,a[=\emptyset$, for all $a\in[0,\infty]$.
\end{assumption}
This class of functions is, in fact, quite wide. In particular, it includes functions that are neither monotonic nor continuous. For instance, 
R\"{u}schendorf and Urusov \cite{ruschendorf2008class} have already considered functions $\Pi$ such that  $0<x_{1l}\leq x_{1r}< x_{2l}\leq x_{2r}<\infty$,  motivated by problems related to Asian options. Regarding applications coming from real options, we can cite, for example, Dixit and Pyndick \cite{DixitPindyck}, Guerra, Nunes and Oliveira \cite{nossoartigo} or Knudsen, Meister and Zervos \cite{knudsen1998valuation}, which analyse the optimal stopping problem for running functions satisfying $0<x_{1l}\leq x_{1r}< x_{2l}\leq x_{2r}=\infty$.

In the following sections, we derive a representation of the value function $V$ for this class of functions.

\section{The optimal strategy}

In this section, we characterize the optimal stopping strategy, $\tau$, for the problem \eqref{optimal00}, under the assumptions presented in the previous sections. For particular running functions, this characterization is obtained by finding the stopping and continuation regions. 
 
We follow the usual approach, namely we start by solving the Hamilton-Jacobi-Bellman (HJB) equation.  As Peskir and Shiryaev \cite{Shiryaev} refer, this is a free-boundary problem. Indeed, there is not a unique function that solves the HJB equation, but instead a family of solutions. However, by using a proper verification theorem, one can completely characterize the unique solution of the optimal stopping problem. 

Since the assumptions considered in section \ref{model} are quite general and different to the ones presented in other works (as the works cited until now), we present a general and suitable verification theorem which give sufficient conditions over the value function, $V$, in order to solve the free-boundary problem above mentioned. Particularly, although we just require \ref{A2} (and, therefore, one may have $v_p^-(x)=\infty$), as we will see later on, this does not affect the regularity of the optimal solution $V$. Notably, the value function $V$ is still continuously  differentiable, with absolutely continuous derivative. We note that in previous works, as  Knudsen, Meister and Zervos \cite{knudsen1998valuation} or R\"{u}schendorf and Urusov \cite{ruschendorf2008class}, a similar result holds (although the conditions that they assume are not the same as ours).

The verification theorems provided by Knudsen, Meister and Zervos \cite{knudsen1998valuation} and by R\"{u}schendorf and Urusov \cite{ruschendorf2008class} are quite general and partially cover the verification theorem presented in this section. Anyhow, we still need to present a suitable verification theorem, as Knudsen, Meister and Zervos \cite{knudsen1998valuation} do not consider  the possibility of having $v_p^-(x)=\infty$ for all $x\in]0,\infty[$, and R\"{u}schendorf and Urusov \cite{ruschendorf2008class}  present a verification theorem only for the case when $0<x_{1l}\leq x_{1r}< x_{2l}\leq x_{2r}<\infty$ in the assumption \ref{A1}.

\subsection{Verification theorem}\label{secao teorema da verificacao}

%
%

%

\begin{teo}\label{teorema da verificacao}
Consider the optimal stopping problem defined   in \eqref{GBM}-\eqref{optimal00}, and the assumptions \ref{A3}-\ref{A2}. Let $v:]0,\infty[\to\mathbb{R}^+$ be a solution to the HJB equation
\begin{equation}\label{HJB}
\min\left\{-{\cal L}v(x)-\Pi(x),v(x)\right\}=0,
\end{equation}
where ${\cal L}$ is given in \eqref{infinitesimal generator GBM}. If $v$ is such that $v\in C^1(]0,\infty[)$, $v'\in AC(]0,\infty[)$ and 
\begin{equation}\label{condi??o auxiliar}
\lim_{t\to\infty}e^{-rt}E_x\left[v(X(t)){\cal I}_{\{\tau_0>t\}}\right]=0,\text{ where } \tau_0=\inf\{t>0:v(X(t))=0\},
\end{equation}
then, (i) the value function $V$ is given by $V(x)=v(x)$ for all $x\in ]0,\infty[$, (ii) the optimal strategy is $\tau^*=\tau_0$ and (iii) the stopping and continuation regions are $D=\{x>0:V(x)=0\}$ and $D^c=\{x>0:V(x)>0\}$, respectively.
\end{teo}

\begin{proof}
Let $v$ be a solution of the HJB equation, such that $v\in C^1(]0,\infty[)$, with $v'\in AC(]0,\infty[)$.
Then, as $v$ is a solution of the HJB, it follows that:
\begin{equation}\label{varational inequalities}
-{\cal L}v(x)=rv(x) - \alpha x v^\prime (x) - \frac{1}{2}\sigma^2x^2 v^{\prime \prime}(x) \geq \Pi(x),
\quad\text{and}\quad
v(x) \geq 0,
\end{equation}
for all $x \in ]0,\infty[$, and at least one of the inequalities is verified as an equality.

Fix $t>0$. Since the function $v'$ is continuous in $]0,\infty[$, then $\int_0^{t\wedge\tau}e^{-rs}v'(X(s))dW(s)$\footnote{$a \wedge b=\min(a,b)$.} is a martingale and, consequently:
\begin{equation}
E_x\left[\int_0^{t\wedge\tau}e^{-rs}v'(X(s))dW(s)\right]=0, \quad \text{for all  }x>0, \tau \in {\cal S}.
\end{equation}
Using a It\^o-Tanaka formula (see, for example Revuz and Yor \cite{revuz2013continuous}, or Ghomrasni and Peskir \cite{ghomrasni2004local}) and 
 the inequalities from \eqref{varational inequalities}, it follows that
\begin{align}
0 &\leq E_x\Big[ e^{-r(t\wedge \tau)}  v(X(t\wedge\tau)) \Big] \nonumber\\&=\nonumber
v(x)-\nonumber
E_x \left[ \int_0^{t\wedge \tau} e^{-rs}\left( r v(X(s)) - \alpha X(s)v^\prime (X(s)) - \frac{1} 2 \sigma^2 X^2(s) v^{ \prime \prime}(X(s)) \right) ds \right]
\\ &\leq 
v(x) - E_x \left[ \int_0^{t\wedge \tau} e^{-rs} \Pi(X(s)) ds \right], \label{inequality ito}
\end{align}
and, therefore, we obtain, for all $\tau\in{\cal S}$:
\begin{align}\label{monotonic convergence theorem}
v(x) \geq E_x \left[ \int_0^{t\wedge \tau} e^{-rs} \Pi^+(X(s)) ds \right]-E_x \left[ \int_0^{t\wedge \tau} e^{-rs} \Pi^-(X(s)) ds \right]. 
\end{align}
Since both  $\{\int_0^{t\wedge \tau} e^{-rs} \Pi^+(X(s)) ds\}_{t\geq 0}$ and $\int_0^{t\wedge \tau} e^{-rs} \Pi^-(X(s)) ds\}_{t\geq 0}$ are non-decreasing sequences of measurable functions (with probability 1) and $E_x \left[ \int_0^{t\wedge \tau} e^{-rs} \Pi^+(X(s)) ds \right]<\infty$, then, by using  the monotonic convergence theorem, it follows that, for all $\tau\in {\cal S}$:
\begin{equation}
\label{Jx}
v(x) \geq E_x \left[ \int_0^{\tau} e^{-rs} \Pi(X(s)) ds \right]=J(x,\tau).
\end{equation}
Thus, next we prove that there exists one stopping time for which the equality in \eqref{Jx} occurs if \eqref{condi??o auxiliar} holds. 
For that purpose, let $\tau_0$ be defined as in \eqref{condi??o auxiliar}; then, with similar arguments as the ones used in order to derive \eqref{monotonic convergence theorem}, it follows that:
\begin{equation}\label{auxiliar-2}
E_x\left[\int_0^{t \wedge \tau_0} e^{-rs}\Pi(X(s)) ds\right]=v(x)-E_x\left[e^{-r{t}} v \left(X\left(t\right)\right){\cal I}_{\{\tau_0> t\}}\right].
\end{equation}
Since the condition \eqref{condi??o auxiliar} holds true, we obtain:
%
\begin{equation}\label{first limit}
\lim_{t\to\infty}E_x\left[\int_0^{t \wedge \tau_0} e^{-rs}\Pi(X(s)) ds\right]=v(x).
\end{equation} 
Therefore, it follows that the optimal value function is equal to $J(x,\tau_0)$ if, and only if, 
\begin{equation}\label{second limit}
\lim_{t\to\infty}E_x\left[\int_0^{\tau_0\wedge t} e^{-rs}\Pi(X(s)) ds\right]=E_x\left[\int_0^{\tau_0} e^{-rs}\Pi(X(s)) ds\right].
\end{equation} 
Finally, by using an argument similar to the one used to prove \eqref{Jx}, \eqref{second limit} is straightforward. Thus, the optimal stopping time is indeed $\tau_0$ and the continuation and stopping regions are as defined  in the statement of the theorem.
\end{proof}

\subsection{Study of the ODE}

According to the verification theorem \eqref{teorema da verificacao},  the value function $V$ in the continuation region will be given by the solution to the associated ordinary differential equation (ODE)
\begin{equation}\label{ode}
-{\cal L}v(x)-\Pi(x)=0,
\end{equation}
where ${\cal L}$ is given in \eqref{infinitesimal generator GBM}, taking into account some boundary conditions. Naturally, the solution to \eqref{ode} may not be classic. For example, if $\Pi$ is discontinuous (which may happen, according to our assumptions), then $v''$ is such that:
$$r\int\int v''(t)dt-\alpha x\int v''(t)dt-\Pi(x)=\frac{1}{2}\sigma^2x^2v''(x),
$$
and, consequently, $v''$  is also discontinuous. 
However, in view of Filippov \cite{filippov}, and under the assumption \ref{A3}, the solution of the ODE, $v$, is such that (i) $v\in C^1$, (ii) depends continuously on the initial condition, and (iii) $v'\in AC$ . Indeed $v$, with the described regularity, is a Caratheodory solution to the ODE \eqref{ode}.

In order to solve the ODE \eqref{ode}, we usually start by proposing a solution to the associated homogeneous equation and, afterwards, we find a particular solution,  $v_p(x)$. Under condition \eqref{usual}, Knudsen, Meister and Zervos \cite{knudsen1998valuation} proved the following probabilistic representation for the particular solution:
\begin{align}\label{particular function}
v_p(x)&=E_x\left[\int_0^{\infty}e^{-rs} \Pi(X(s)) ds\right]\nonumber\\
&=\frac{2}{\sigma^2(d_2-d_1)}\left[ x^{d_1}\int_0^{x}s^{-d_1-1}\Pi(s)ds+x^{d_2}\int_x^{\infty}s^{-d_2-1}\Pi(s)ds\right].
\end{align}
We refer also to Kobila \cite{kobila1993}  to complement the discussion about this representation, and Johnson and Zervos \cite{johnson2007solution} for similar results under more general diffusion processes.

In view of our assumptions regarding $\Pi$, condition \eqref{usual} does not need to hold. Therefore, one needs to derive the solution of the ODE using another approach, that we describe next. Let $x^* \in \mathbb{R}^+$ be a point such that  $v(x^*)$ is an initial condition for $v$ and consider the following change of variable:  $t=ln\left( \frac{x}{x^*} \right)$ and $u(t)=v(x^*e^t)$. Then, finding a solution to \eqref{ode} is equivalent to finding a solution to the equation
\begin{equation}\label{ode_modified}
-{\cal\tilde{ L}}u(t)-\Pi(e^t)=0,
\end{equation}
where ${\cal\tilde{ L}}$ is defined as in \eqref{modified infinitesimal generator}.
Now, defining the vector $w(t)=(u(t), u'(t))^T $
where the symbol $^T$ denotes the transpose, we may represent the ODE \eqref{ode_modified} as  $w'(t)=Aw(t)+b(t)$, where $b:]0,\infty[\to\mathbb{R}^{2\times 1}$ is a vector function and $A$ is a constant $2 \times 2$ matrix,  defined as follows:
\begin{equation}
b(t)= -\frac{2}{\sigma^2}\left(0, -\frac{2}{\sigma^2}\Pi(e^t)\right)^T; 
\quad
A=\left(\!
    \begin{array}{cc}
      0 & 1 \\
      \frac{2r}{\sigma^2}&-\frac{\sigma^2-2\alpha}{\sigma^2}
    \end{array}
  \!\right)=\left(\!
    \begin{array}{cc}
      0 & 1 \\
      -d_1d_2&d_1+d_2
    \end{array}
  \!\right).
  \label{ab}
  \end{equation}
The last equality follows from the parametrization defined in \eqref{reparametrization}. Furthermore, straightforward calculations lead to the fundamental matrix  
$$
e^{tA}=\left(\!
    \begin{array}{cc}
      \frac{d_2e^{td_1}-d_1e^{td_2}}{d_2-d_1} & \frac{e^{td_2}-e^{td_1}}{d_2-d_1} \\
      -d_1d_2\frac{e^{td_2}-e^{td_1}}{d_2-d_1}&\frac{d_2e^{td_2}-d_1e^{td_1}}{d_2-d_1}
    \end{array}
  \!\right).
$$
The solution for this system is given by $w(t)=e^{tA}w_0-\frac{2}{\sigma^2}\int_0^te^{(t-s)A}b(s)ds$, where $w_0=w(0)$ represents the initial condition. Returning to the original variables:
\begin{align}\label{v}
&v(x)=\frac{-2}{\sigma^2(d_2-d_1)}\left(\left(\frac{x}{x^*}\right)^{d_1}(d_2A_1-A_2)+\left(\frac{x}{x^*}\right)^{d_2}(A_2-d_1A_1)+\int_{x*}^x\frac{\left(\frac{x}{s}\right)^{d_2}-\left(\frac{x}{s}\right)^{d_1}}{s}\Pi(s)ds\right)\\\label{v'}
&v'(x)=\frac{-2x^{-1}}{\sigma^2(d_2-d_1)}\left(\left(\frac{x}{x^*}\right)^{d_1}(d_2A_1-A_2)d_1+\left(\frac{x}{x^*}\right)^{d_2}(A_2-d_1A_1)d_2+\int_{x*}^x\frac{\left(\frac{x}{s}\right)^{d_2}-\left(\frac{x}{s}\right)^{d_1}}{s}\Pi(s)ds\right),
\end{align}
where $A_1=\frac{\sigma^2}{2}v(x^*)$ and $A_2=\frac{\sigma^2}{2}v'(x^*)$.

\begin{remark}\label{solution to the ode}
When $x^*=X(\tau^*)$, where $\tau^*$ is defined in the verification theorem \eqref{teorema da verificacao}, then it follows that $A_1=A_2=0$, and, therefore, we guess that the solution to the ODE \eqref{ode}, $v$, is given by:
\begin{align}  
    \label{v and v'}
 \left(\!
    \begin{array}{c}
      v(x) \\
      v'(x)
    \end{array}\!\right) &=\frac{-2}{\sigma^2(d_2-d_1)}\left(\!
    \begin{array}{c}
      \int_{x^*}^x\frac{\left(\frac{x}{s}\right)^{d_2}-\left(\frac{x}{s}\right)^{d_1}}{s}\Pi(s)ds \\
     \frac{1}{x} \int_{x^*}^x\frac{d_2\left(\frac{x}{s}\right)^{d_2}-d_1\left(\frac{x}{s}\right)^{d_1}}{s}\Pi(s)ds
    \end{array}
  \!\right).
\end{align} 
\end{remark}

%
%
%
%
%

\section{Solution to the optimal stopping problem}\label{Solution of the Optimal Stopping Time}

In this section we provide a closed form analytic solution to the optimal stopping time problem \eqref{optimal00}, under assumptions \ref{A3}-\ref{A1}.
To ease the presentation, first we  discuss important aspects about the main result of this section, and, afterwards, we provide it.

According to the comments in the previous sections, we expect that the solution of \eqref{optimal00} is given by:
\begin{equation}\label{value function}
V(x)=\begin{cases}
0, & \quad x\in {\cal D}_{\Pi}\\
\frac{-2}{\sigma^2(d_2-d_1)}\int_{x^*}^x\frac{\left(\frac{x}{s}\right)^{d_2}-\left(\frac{x}{s}\right)^{d_1}}{s}\Pi(s)ds & \quad x\in {\cal D}^c_{\Pi}
\end{cases}.
\end{equation}
In \eqref{value function}, ${\cal D}_{\Pi}$ is the stopping region and ${\cal D}^c_{\Pi}$ is the continuation region. We use the notation ${\cal D}_{\Pi}$ and ${\cal D}^c_{\Pi}$ to emphasize that the unknown sets are strictly related to $\Pi$. Henceforward, we will use ${\cal D}$ instead of ${\cal D}_\Pi$ and ${\cal D}^c$ instead of ${\cal D}^c_\Pi$ to ease the notation.

In view of the theorem \ref{teorema da verificacao}, if $V$ is the solution of \eqref{optimal00},  then it is also a continuous differentiable solution to the corresponding HJB equation \eqref{HJB}, with absolute continuous derivative. Therefore, in order to meet these conditions, one uses the so-called \textit{pasting conditions}, from which one intends to find a unique $V$ and, consequently, unique ${\cal D}$ and ${\cal D}^c$. 

Additionally, the regions ${\cal D}$ and ${\cal D}^c$ may have different forms, depending on the shape of the function $\Pi$. In general, they are also unknown. For instance, if $\Pi$ is monotonically increasing, we guess that ${\cal D}=]0,x^*]$, for some $x^*>0$. However, if $\Pi$ is monotonically decreasing, the stopping region should be given by ${\cal D}=[x^*,\infty[$, for some $x^*>0$. 
As mentioned in the previous section, such problems are called {\em free boundary problems}. One of the techniques used to solve such problems is based on truncation methods as Guerra, Nunes and Oliveira \cite{nossoartigo} propose. Next, we present the definition of $x^*$ according to the class of profit functions $\Pi$ that we are using. Before presenting the result, we remark that the type of relations that we find in definition \eqref{def of x*} are also found when one uses truncation methods.
\begin{defi}\label{def of x*}
Let $x^*$ be such that $x^*\equiv x^*_\Pi\in[0,\infty]$, where $\Pi$ is a running function and where ${\cal D},{\cal D}^c\subset]0,\infty[$ are two families of regions, such that ${\cal D}_{x^*}\cap{\cal D}^c_{x^*}=\emptyset$ and ${\cal D}_{x^*}\cup{\cal D}^c_{x^*}=]0,\infty[$. For each function $\Pi$, we define $x^*$ as follows:
\begin{itemize}
\item[a)] if, in assumption \ref{A1}, $0<x_{1l}<x_{2r}=\infty$, then 
\begin{equation}\label{gamma condition}
x^*=\gamma=\inf\left\{x>0:\exists C>x, \int_x^C\frac{\left(\frac{C}{s}\right)^{d_2}-\left(\frac{C}{s}\right)^{d_1}}{s}\Pi(s)ds\geq 0\right\} \text{ and } {\cal D}=]0,\gamma],\text{ } {\cal D}^c=]\gamma,\infty[;
\end{equation}
 \item[b)] if, in assumption \ref{A1}, $0=x_{1l}<x_{2r}<\infty$, then 
\begin{equation}\label{zeta condition}
x^*=\zeta=\sup\left\{x>0:\exists C\in]0,x[, \int_C^x\frac{\left(\frac{C}{s}\right)^{d_2}-\left(\frac{C}{s}\right)^{d_1}}{s}\Pi(s)ds\geq 0\right\} \text{ and }  {\cal D}=[\zeta,\infty[,\text{ } {\cal D}^c=]0,\zeta[;
\end{equation}
\item[c)] if, in assumption \ref{A1}, $0<x_{1l}<x_{2r}<\infty$,
\begin{align}\label{delta beta condition}
\begin{cases}&\delta=\inf\left\{x\in]0,\beta[: \int_x^\beta\frac{\left(\frac{\beta}{s}\right)^{d_2}-\left(\frac{\beta}{s}\right)^{d_1}}{s}\Pi(s)ds\geq 0\right\}\\
&\beta=\sup\left\{x>\delta:\int_\delta^x\frac{\left(\frac{\delta}{s}\right)^{d_2}-\left(\frac{\delta}{s}\right)^{d_1}}{s}\Pi(s)ds\geq 0\right\}\\
&{\cal D}=]0,\delta]\cup[\beta,\infty[, \text{ }{\cal D}^c=]\delta,\beta[, \text{ }x^*=\delta \text{ or }x^*=\beta.\end{cases}
\end{align}
We note that, given the definition of $\delta$ and $\beta$, it is indifferent, in terms of $V$, which one we use as $x^*$.
 \end{itemize}
 \end{defi}
We notice that $\gamma,\zeta,\delta$ and $\beta$ are defined for  different functions. Therefore, for a fixed function $\Pi$ satisfying assumption \ref{A1}, just one of the definitions  \eqref{gamma condition}-\eqref{delta beta condition} makes sense.

Usually, we cannot guarantee to have $\gamma,\delta>0$ and $\zeta,\beta<\infty$, meaning that it may never be optimal to stop the process for any initial condition $x>0$. For instance, fixing $\alpha=0.3$, $\sigma^2=0.1$ and $r=-0.1$ and considering
\begin{equation*}
\Pi(x)=\begin{cases}
(x+1)(x-1),& x\leq 3\\
\frac{72}{x^2},& x>3,
\end{cases}
\end{equation*}
we have, $v_p^+(x)<\infty$ and $J(x,\infty)>0$, for all $x>0$. Consequently, we would have $\gamma=0$ and a continuation region given by ${\cal D}=]0,\infty[$.
Therefore, it would never be optimal to stop the process, and the value function would be given by
\begin{equation}\label{trivial value function}
 V(x)=J(x,\infty), \quad \text{for all }x>0.
\end{equation}
According to proposition 4.1 in Knudsen, Meister and Zervos \cite{knudsen1998valuation}, we have  $J(x,\infty)=v_p(x)$, with $v_p(x)$ defined as in \eqref{particular function}, if, in addition to assumption \ref{A2}, $v_p^-(x)<\infty$, for all $x\in]0,\infty[$.

Often, in this kind of stochastic models (see for example Knudsen, Meister and Zervos \cite{knudsen1998valuation}, Duckworth and  Zervos \cite{duckworth2000investment}), one represents the threshold $x^*$ as the unique solution to integral equations. Due to the possibility of having $\gamma=0$ or $\delta=0$ or $\zeta=\infty$ or $\beta=\infty$ this could not be true in our set-up. In the next proposition, we prove that, in some situations, the threshold $x^*$ may be defined by the following equations:
\begin{align}
\label{gamma eq}
& \int_{\gamma}^{\infty}s^{-d_2-1}\Pi(s)ds=0,\\
\label{zeta eq}
& \int_0^{\zeta}s^{-d_1-1}\Pi(s)ds=0,\\
\label{zeta, delta eq}
&\begin{cases} \int_\delta^\beta s^{-d_1-1}\Pi(s)ds=0,\\\int_\delta^\beta s^{-d_2-1}\Pi(s)ds=0.\end{cases}
 \end{align}
We emphasize that these equations make sense for different running functions $\Pi$, as one can see in definition \ref{def of x*}. 
\begin{prop}\label{x^* equation}
Let $x^*, \gamma, \zeta,\delta$ and $\beta$ be defined as in definition \ref{def of x*} and assume that $x^*=\gamma>0$ or $x^*=\zeta<\infty$ or both $x^*=\delta>0$ and $\beta<\infty$ hold true. Then, $x^*$ may be equivalently defined as the unique solution to one of the equations \eqref{gamma eq}-\eqref{zeta, delta eq}.
\end{prop}
\begin{proof}
If $x^*=\gamma>0$ or $x^*=\zeta<\infty$, then $x^*$ may also be defined as the unique solution to
\begin{align*}
&\lim\limits_{C\to\infty}\left(\int_{\gamma}^{C}s^{-d_2-1}\Pi(s)ds-{C}^{d_1-d_2}\int_{\gamma}^{C}s^{-d_1-1}\Pi(s)ds\right)=0,\\
&\lim\limits_{C\downarrow 0}\left({C}^{d_2-d_1}\int_{C}^{\zeta}s^{-d_2-1}\Pi(s)ds-\int_{C}^{\zeta}s^{-d_1-1}\Pi(s)ds\right)=0.
\end{align*}
Moreover, these equations combined with the results in proposition \ref{auxiliar para condicoes integrais}, allow us to obtain \eqref{zeta eq}-\eqref{gamma eq}. Finally, assuming that $x^*$ is such that $\delta>0$ and $\beta<\infty$, then the intended result follows by solving a system of two equations.
\end{proof}
When $\delta$ and $\beta$ are such that $\delta>0$ and $\beta=\infty$ or $\delta=0$ and $\beta<\infty$, then the definition of $\delta$, in the first case, and the definition of $\beta$, in the second case, degenerates, respectively, in the conditions \eqref{gamma condition} and \eqref{zeta condition}.

At this point, it is imperative to know how the model's data influences the continuation region. The next proposition gives us the set of parameters for which ${\cal D}\varsubsetneq]0,\infty[$ holds true.
\begin{prop}\label{x^* diferente de zero}
Let $x^*, \gamma, \zeta,\delta$ and $\beta$ be defined as in definition \ref{def of x*}. Assuming that $\Pi$ is such that:
\begin{itemize}
\item[1)] in assumption \ref{A1}, $0<x_{1l}<x_{2r}=\infty$ or $0<x_{1l}<x_{2r}<\infty$, there is $\epsilon>0$ and $k>0$, such that $\Pi(x)<-k$, for all $x\in]0,\epsilon[$ and, finally, either $d_2\geq 0$ or both $d_2< 0 \text{ and } v_p^-(x)=\infty \text{ for all } x>0 $. Then, one of the conditions $x^*=\gamma>0$ or $x^*=\delta>0$ holds true.
\item[2)] in assumption \ref{A1}, $0=x_{1l}<x_{2r}<\infty$ or $0<x_{1l}<x_{2r}<\infty$, there is $M>0$ and $k>0$, such that $\Pi(x)<-k$, for all $x\in]M,\infty[$ and, finally, either $d_1\leq 0$ or both $d_1> 0 \text{ and } v_p^-(x)=\infty \text{ for all } x>0 $. Then, one of the conditions $x^*=\zeta<\infty$ or $x^*=\beta<\infty$ holds true.
\end{itemize}
\end{prop}
\begin{proof}
In light of remark \ref{l1}, the proof is straightforward when we assume that  $v_p^-(x)=\infty$. Furthermore, taking into account propositions \ref{x^* equation} and \ref{auxiliar para condicoes integrais}, the result in 1) and 2) follows respectively from:
\begin{align}
&\lim_{x\downarrow 0}\int_x^{\nu}s^{-d_2-1}\Pi(s)ds\leq\int_\epsilon^{\nu}s^{-d_2-1}\Pi(s)ds-\lim_{x\downarrow 0}\int_x^{\epsilon}ks^{-d_2-1}ds=-\infty,\\
&\lim_{x\to \infty}\int_\nu^{x}s^{-d_1-1}\Pi(s)ds\leq\int_\nu^{M}s^{-d_1-1}\Pi(s)ds-\lim_{x\to \infty}\int_M^{x}ks^{-d_1-1}ds=\infty.
\end{align} 
\end{proof}
When the conditions in proposition \ref{x^* equation} are not satisfied, it is impossible to know \textit{a priori} what kind of continuation region $\left({\cal D}=]0,\infty[ \text{ or } {\cal D}\varsubsetneq]0,\infty[\right)$ is expected, in particular  when $r<0$. Even when $r\geq 0$, if $\Pi(x)\to 0$ as $x\downarrow 0$ or $x\to \infty$, it may be impossible to answer this question \textit{a priori}. To exemplify these questions, suppose that $\Pi:]0,\infty[\to\mathbb{R}$ is a sufficiently smooth function, satisfying $0<x_{1l}<x_{2r}=\infty$ in assumption \ref{A1}, and such that $\Pi(x)\to 0$ as $x\downarrow 0$. If $\Pi^{(n)}(0)=\lim\limits_{x\downarrow 0}\Pi^{(n)}(x)$, where  $\Pi^{(n)}$ represents the $n$-th derivative of $\Pi$ and $n=\inf\{n\in\mathbb{N}:\Pi^{(n)}(0)\neq 0\}$, then the improper integral $\int_0^{\nu}s^{-d_2-1}\Pi(s)ds$ is convergent, depending on the sign of $n-d_2$.

Now, we are able to present the main result of this section. 

\begin{teo}\label{teo importante}
Let $V$ be the value function defined as in \eqref{optimal00} and $x^*, \gamma, \zeta,\delta, \beta, {\cal D}$ and ${\cal D}^c$ defined as in definition \ref{def of x*}.
\begin{itemize}
\item[1)] If one of the conditions $x^*=\gamma>0$, $x^*=\zeta<\infty$ or either $x^*=\delta>0$ or $x^*=\beta<\infty$ holds true, then $V$ is given by \eqref{value function} and ${\cal D}$ and ${\cal D}^c$ are indeed the stopping and continuation regions, respectively;
\item[2)] otherwise, $V(x)=v_p(x)$, for all $x>0$, where $v_p$ is given by \eqref{particular function}. Furthermore, ${\cal D}^c=]0,\infty[$.
\end{itemize}
\end{teo}
\begin{proof}
The proof of this result follows directly from proposition \ref{HJB eq solution}, corollary \ref{envelopes para v} and lemmas \ref{A.1} and \ref{A.2}. Moreover, the proof of \eqref{condi??o auxiliar} in case 2) can be found in  Knudsen, Meister and Zervos \cite{knudsen1998valuation}.
\end{proof}

\begin{remark} 
If $v_p^-(x)<\infty$, the value function \eqref{value function} takes the following familiar representation:
 \begin{align*}
 v(x)=
 \begin{cases}
a_1x^{d_1}+a_2x^{d_2}+v_p(x),& x\in D^c\\
0, & x\in D.
\end{cases}
 \end{align*}
 with $v_p$ given by \eqref{particular function}, $a_1=\frac{-2}{\sigma^2(d_2-d_1)}\int_{0}^{x^*} s^{-d_1-1}\Pi(s)ds$ and $a_2=\frac{-2}{\sigma^2(d_2-d_1)}\int_{x^*}^{\infty} s^{-d_2-1}\Pi(s)ds$. 
This representation is widely used in the literature of real options (see e.g. Knudsen, Meister and Zervos \cite{knudsen1998valuation} and Kobila \cite{kobila1993}). For particular choices of $\Pi$, it is possible to explicitly compute the involved integrals  (see e.g. Dixit and Pindyck \cite{DixitPindyck}, Guerra, Nunes and Oliveira \cite{nossoartigo} and Trigeorgis \cite{trigeorgis}). Moreover, according to \eqref{zeta eq}-\eqref{gamma eq}, $a_1,a_2$ may (or may not) be equal to $0$. 

\end{remark}

As we have introduced in the beginning of this paper, the optimal entrance time problem is strictly related to the problem discussed until know. Indeed, if condition \eqref{usual} holds true, it is possible to write
$$\sup_{\tau\in{\cal S}}L(x,\tau)=J(x,\infty)+\sup_{\tau\in{\cal S}}(-J(x,\tau))=v_p(x)+\sup_{\tau\in{\cal S}}(-J(x,\tau)),$$
with $v_p$ is defined in \eqref{particular function}. Therefore, and according to theorem  \ref{teo importante}, there exists one stopping time $\tau^*\in{\cal S}$, such that both functional
$$
E_x\left[\int_{\tau}^{\infty}e^{-r s}\Pi(X(s))ds\right]\quad \text{and}\quad E_x\left[\int_0^{\tau}e^{-r s}\left(-\Pi(X(s))\right)ds\right]
$$
are maximized when $\tau=\tau^*$. Moreover, the stopping and continuation regions for problem \eqref{entry problem}, when one considers the running function $\Pi$, are, respectively, the stopping and continuation regions for the optimal stopping problem \eqref{optimal00}, for the running function $-\Pi$. 

Here, we simply have a well-defined problem if both conditions $v_p^+(x)<\infty$ and $v_p^-(x)<\infty$ hold true. In fact, on the one hand, if $v_p^+(x)=\infty$ for all $x\in\mathbb{R}^+$, $L(0,x)=\infty$ and, thus, the problem is trivial: the entrance should occur immediately. On the other hand, if $v_p^-(x)=\infty$ for all $x\in\mathbb{R}^+$, $L(\tau,x)=-\infty$, for all $\tau$ and, therefore, the entrance should never occur.

We make the following remark, in order to avoid misunderstandings in the interpretations of the next sections, particularly in the section of the Sensitivity Analysis: 
\begin{remark} 
In real options, it is common to assume, in both exit and investment options, having  profit functions, such that, in assumption \ref{A1}, $0<x_{1l}<x_{2r}=\infty$ holds true. In light of the comments made concerning the relation between the optimal stopping time problem and the optimal entrance time problem, we will associate, without further references, the threshold $\gamma$ with the exit option and the threshold $\zeta$ with the investment option. 
\end{remark}

\section{Sensitivity analysis}\label{Sensitivity Analysis}
In this section, we study the behaviour of the stopping strategy concerning the diffusion parameters $\alpha$ and $\sigma^2$. As we have previously stated, we are able to acommodate several non-standard assumptions, as non-monotonic payoff functions and negative interest rates, in our analysis. Contributions like  Dixit and Pindyck \cite{DixitPindyck}, Guerra, Nunes and Oliveira \cite{nossoartigo} and Trigeorgis \cite{trigeorgis} consider more restrictive cases, with, for example, positive interest rates and polynomial increasing pay-off functions. Here, we need to distinguish two cases:
\begin{itemize}
\item[(i)] The profit function $\Pi$ changes its sign just once (and, thus, either $0<x_{1l}<x_{2r}=\infty$ or $0=x_{1l}<x_{2r}<\infty$). In that case, the continuation and stopping regions are two complementary sets, and the relevant thresholds are $\gamma$ and $\zeta$, respectively. In order to emphasize how $\gamma$ and $\zeta$ depend on the parameters $\alpha$ and $\sigma^2$, we introduce the notation $\gamma(\alpha,\sigma^2)$ and $\zeta(\alpha,\sigma^2)$. As this will help to better understand the results, we will associate the continuation region $(\gamma,\infty)$ to the decision to exit the market, and the continuation region $(0,\zeta)$ to the investment decision. Consequently, in this setting, if $\gamma$ increases/decreases then the exit decision is antecipated/postponed, whereas if $\zeta$ increases/decreases the investment decision is postponed/antecipated.

\item[(ii)] The profit function changes its sign twice (and, in that case, we have $0<x_{1l}<x_{2r}<\infty$). Here, we were not able to derive analytical results concerning  the behaviour of the two thresholds $\delta$ and $\beta$. In section \ref{exemplo} we provide a numerical illustration, and, in particular, we check, numerically, what is the influence of $\alpha$ and $\sigma^2$ on these thresholds.
\end{itemize}

Next, we present the results derived for case (i). We note that, when the interest rate $r$ is positive and the pay-off function is monotonic (increasing or decreasing), the results that we derive are the ones that we find in the literature. One of the punch-marks of this paper is precisely the results concerning the behaviour of the thresholds with the drift $\alpha$ and the volatility $\sigma^2$ for negative interest rate $r$, as they are quite different from the standard ones. In particular we lose, in some cases, the monotonicity of the thresholds. 

The interpretation and validity of these results is far from being trivial and general, and relies on the type of pay-off function $\Pi$ that we are considering. For example, if one assumes an increasing pay-off function, negative interest rates and $\frac{\sigma^2}{2}-\alpha<0$ lead to trivial problems. Therefore, in this section, we present the results regarding the behaviour of the thresholds with respect to $\alpha$ and $\sigma^2$ without an exhaustive interpretation, keeping in the back of our minds that some choices of $r$ and $\alpha$ may not be possible for the particular pay-off function under consideration. Thus, the application of these results to a concrete case must be carefully analysed. We will come back to this question in section \ref{exemplo}, where we will present an illustrative example.

In the following tables we summarize the results that we found for the behaviour of the decision thresholds. We stress that the signs of the derivatives of the thresholds concerning the parameters have exactly the same signs of the derivatives of the characteristic polynomial roots in order to the same parameters, as one can see in the next proposition and in proposition \ref{raizes decrescentes}.

\begin{prop}
Consider the optimal stopping problem \eqref{optimal00} under the assumptions \ref{A3}-\ref{A1} and the thresholds $\gamma$ and $\zeta$ as in \eqref{gamma eq}-\eqref{zeta eq}. Then, the functions $\gamma(\alpha,\sigma^2)$ and $\zeta(\alpha,\sigma^2)$ are such that the information in the tables \eqref{table 1}-\eqref{table 3} is verified.
\end{prop}
\begin{proof}
Using the implicit differentiation rule in the equations \eqref{gamma eq}-\eqref{zeta eq}, we derive the following: 
\begin{align*}
&\frac{\partial \gamma}{\partial i}(\alpha,\sigma^2)=-\frac{\partial d_2}{\partial i}\frac{\int_{\gamma}^{\infty}s^{-d_2-1}\log(s)\Pi(s)ds}{{(\gamma)}^{-d_2-1}\Pi(\gamma)},\quad \text{with }i=\alpha\text{ or }\sigma^2,\\
&\frac{\partial \zeta}{\partial i}(\alpha,\sigma^2)=\frac{\partial d_1}{\partial i}\frac{\int_{0}^{\zeta}s^{-d_1-1}\log(s)\Pi(s)ds}{{(\zeta)}^{-d_1-1}\Pi(\zeta)},\quad \text{with }i=\alpha\text{ or }\sigma^2.
\end{align*}
The result follows, in view of the lemmas \ref{raizes decrescentes} and \ref{analise real}, and from the fact that $\Pi(\gamma)<0$ and $\Pi(\zeta)<0$. 
\end{proof}
 
In table \ref{table 1} we study the behaviour of $\zeta$ and $\gamma$ with the drift $\alpha$, assuming that $\sigma^2$ is fixed. We note that, in each of the following tables, we simply cover the parameters' domain for which $d_1<d_2\in\mathbb{R}$.

\begin{table}[h!]
\centering
\caption{Sign of the derivatives of $\zeta(.,\sigma^2)$ and $\gamma(.,\sigma^2)$, as well as $d_1(., \sigma^2)$ and $d_2(., \sigma^2)$.}
\label{table 1}
\begin{tabular}{l|c|c|c|c|c|}
\cline{2-6}
                                                     & \multicolumn{2}{c|}{$r<0$}                & \multicolumn{2}{c|}{$r=0$}                & $r>0$ \\ \cline{2-6} 
                                                     & $\alpha<\sigma^2/2$ & $\alpha>\sigma^2/2$ & $\alpha<\sigma^2/2$ & $\alpha>\sigma^2/2$ &       \\ \hline
\multicolumn{1}{|l|}{$\partial d_1/\partial \alpha$ and $\partial \zeta/\partial \alpha$} & $>0$                & $<0$                & $=0$                & $<0$                & $<0$  \\ \hline
\multicolumn{1}{|l|}{$\partial d_2/\partial \alpha$ and $\partial \gamma/\partial \alpha$} & $<0$                & $>0$                & $<0$                & $=0$                & $<0$  \\ \hline
\end{tabular}
\end{table}

Therefore, for positive interest rates, $r$, we obtain the known result: increasing the drift has the effect of postponing the exit and anticipating the investment decision. On the contrary, for negative $r$, the result is quite different: for sufficiently {\em large} values of $\alpha$ (when $\alpha>\frac{\sigma^2}{2}$) both decisions are antecipated whereas for {\em small} values of $\alpha$ both decisions are postponed. Furthermore, if the interest rate is zero,   {\em small} values of the drift do not have any impact on the decision regarding investment, but they impact on the exit decision, and the abandonment is postponed. In case of {\em large} values of the drift, the exit decision is not affected by changing drift, but the investment decision occurs earlier.

In Tables \ref{table 2} and \ref{table 3}, we analyse the monotonicity of $\gamma$ and $\zeta$ in order to $\sigma^2$, with $\alpha$ fixed.

\begin{table}[h!]
\centering
\caption{Sign of the derivatives of $\zeta(\alpha,.)$ and $d_1(\alpha,.)$.}
\label{table 2}
\begin{tabular}{|c|c|c|c|c|}
\hline
\multicolumn{2}{|c|}{$r<0$}                       & \multicolumn{2}{c|}{$r=0$}                        & $r>0$ \\ \hline
$\alpha<{\sigma^2}/{2}$ & $\alpha>{\sigma^2}/{2}$ & $\alpha<{\sigma^2}/{2}$ & $\alpha>{\sigma^2}/{2}$ &       \\ \hline
$<0$                    & $>0$                    & $=0$                    & $>0$                    & $>0$  \\ \hline
\end{tabular}
\end{table}
\vspace{0.5cm}

\begin{table}[h!]
\centering
\caption{Sign of the derivatives of $\gamma(\alpha,.)$ and $d_2(\alpha,.)$.}
\label{table 3}
\begin{tabular}{|c|c|c|c|c|c|c|c|c|c|}
\hline
\multicolumn{4}{|c|}{$r<0$}                                              & \multicolumn{3}{c|}{$r=0$}                               & \multicolumn{3}{c|}{$r>0$}           \\ \hline
$\alpha<r$ & $\alpha=r$ & $r\leq\alpha<\sigma^2/2$ & $\alpha>\sigma^2/2$ & $\alpha<0$ & $0<\alpha<\sigma^2/2$ & $\alpha>\sigma^2/2$ & $\alpha<r$ & $\alpha=r$ & $\alpha>r$ \\ \hline
$<0$       & =0         & \textgreater0            & $<0$                & $<0$       & $>0$                  & =0                  & $<0$       & $=0$       & $>0$       \\ \hline
\end{tabular}
\end{table}

The analysis of both tables reveals two striking facts: (i) the interest rate $r$ influences the behaviour of $\zeta$ and $\gamma$ differently when $\sigma^2$ varies, and (ii) the relative position of $\alpha$ relatively to $r$ and $\sigma^2$ completely changes how $\sigma^2$ influences $\zeta$ and $\gamma$. When $r<0$, the impact of $\alpha$ in $\zeta$ depends only on the relation between $\alpha$ and $\sigma^2$ (see table \ref{table 2}). For $\gamma$, however, this impact depends on the relation between $\alpha$ and both $\sigma^2$ and $r$ (see table \ref{table 3}). Also, we observe the following behaviours:
\begin{itemize}
\item In case $r>0$ and $\alpha<r$, we obtain the usual result: increasing volatility postpones the decision (either to invest or to exit). But in case $\alpha>r$ is allowed (which depends on the behaviour of $\Pi$), then the exit decision is anticipated and the investment decision is postponed.
\item If $r=0$, the investment decision is not affected with increasing $\sigma^2$ as long as $\alpha<\frac{\sigma^2}{2}$; if $\alpha>\frac{\sigma^2}{2}$, the investment decision is postponed. The behaviour of the exit decision is quite different: on the one hand, if $\alpha>\frac{\sigma^2}{2}$, it does not change with $\sigma^2$; on the other hand, if $\alpha<\frac{\sigma^2}{2}$, it depends on the sign of $\alpha$: if $\alpha>0$, the exit decision is anticipated, otherwise, it is postponed.
\item If $r<0$, regarding the investment decision, increasing $\sigma^2$ will anticipate it, as long as $\alpha<\frac{\sigma^2}{2}$; contrarily, the decision will occur later. Concerning the exit decision, it depends on the relation between $\alpha, \sigma^2$ and $r$. For $\alpha<r$ or $\alpha>\frac{\sigma^2}{2}$, the abandonment is postponed; conversely, the decision does not change if $\alpha=r$ and it is anticipated otherwise. 
\end{itemize}

\section{Example: non-monotonic profit functions}\label{exemplo}
In this section, we use the results previously derived, when $\Pi$ has a particular behaviour (notably, it is non-monotonic). The idea behind the example comes from a paper of Dahan and Srinivasan \cite{dahan2011impact}, where it contains some considerations about the consequences of cost reduction on gross profit. 

For product managers, the study of gross profit is crucial. Indeed, it is specially important in questions related to the impact of unit manufacturing costs, when the revenues of a firm are highly price sensitive (Griffin and Hauser \cite{Griffin}, Kahn \cite{Kahn}).

Following Dahan and Srinivasan \cite{dahan2011impact}, a gross profit function $\Pi$ is defined as follows:
\begin{equation*}
\Pi(x)=x\times q(x)-TVC(x),
\end{equation*}
where $x$ denotes the unit price of the asset, $q(x)$ is the demanded quantity, which is a function of the price $x$, and $TVC(x)$ is the {\em total variable cost} (eventually, also a function of the price $x$).  In case of proportional costs, it follows that $TVC(x)=c\times x$, where $c$ denotes the unit (variable) cost. For any given unit cost $c$,  the gross profit function $\Pi$ is strictly quasi concave and it is a smooth function of $x$. In particular, this function increases for a certain range of prices, from $0$ to $x_{\max}$, and then decreases, where $x_{max}$ is the unique profit-maximizing price, at which point the function is locally strictly concave.

In order to proceed with the illustration of the above obtained results, we assume that the firm takes price as given, and that the price process, $X$, follows a geometric Brownian motion, as in \eqref{GBM}. Additionally, we assume a particular instance of a function $\Pi$ with the described characteristics. Notably, we consider $\Pi$ as follows:
\begin{equation}\label{pi example}
\Pi(x)=\begin{cases}
-c(x-a)(x-b), & x\leq x_0\\
\frac{e}{x-f}+K, & x> x_0,
\end{cases}
\end{equation}
where $a,b,c,d,f\geq 0$ and $K\in \mathbb{R}$, and with 
\begin{align}\label{x0}
&x_0=\frac{1}{3}\left(a+b+f+\sqrt{(a+b+f)^2-3\left(ab+(a+b)f+\frac{K}{c}\right)}\right),\\
\label{e}
&e=c(2x_0-a-b)(x_0-f)^2.
\end{align}
The equalities \eqref{x0}-\eqref{e} follow, since $\Pi\in C^1(]0,\infty[)$. Here, $K$ is the "guaranteed" gross profit, for high levels of unit price. In the next two sections, we  solve the optimal stopping problem \eqref{optimal00} for $K\geq 0 $ and $K<0$, while assuming particular values for the involved parameters.

\subsection{One side stopping region : $K\geq 0$}
In this subsection, we assume that: $a=1,\text{ } b=10,\text{ }c=1,\text{ }f=2$ and $K=4$.  We note that $v_p^+(x)<\infty$, for all $x>0$  if, and only if, $r>\sigma^2-\alpha$. Therefore, for the sake of illustration, and in order to meet these conditions, we fix $r=\alpha=\sigma^2=0.1$. In figure \ref{figure 1}, we plot $\Pi$ as a function of $x$. 

\begin{figure}[!h]\center
\caption{\label{figure 1}A gross profit function $\Pi$ when $K\geq 0$.}
\vspace{0.5cm}
\includegraphics[scale=0.7]{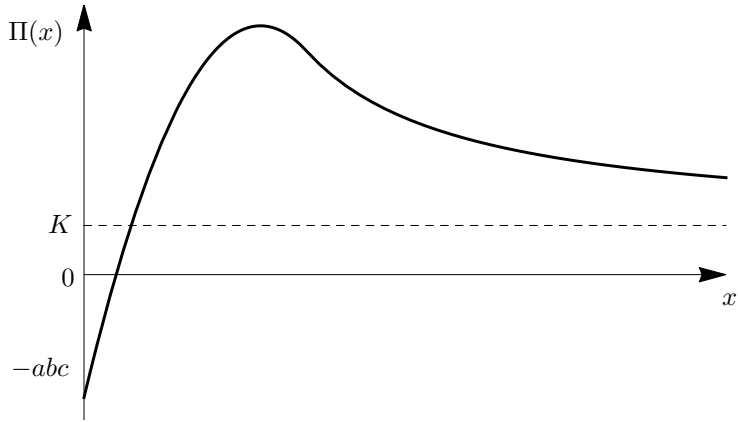}

\begin{picture}(160,10)

\put(-55,82){$K$}
\put(-70,155){$\Pi(x)$}
\put(200,55){$x$}
\put(-50,62){$0$}
\put(-69,28){$-abc$}
\end{picture}
\end{figure}

As  $\Pi<0$, for $x<1$, we expect that $D^c=]\gamma,\infty[$, with $\gamma$ being the solution of the following integral equation:
\begin{equation}
\int_{\gamma}^{x_0}-cs^{-2}(s-1)(s-10)ds+\int_{x_0}^{\infty}s^{-2}\left(\frac{e}{s-2}+4\right)ds=0\Rightarrow \gamma\simeq 0.3384,
\end{equation} 
where $x_0$ and $e$ are given as in  \eqref{x0}-\eqref{e}. Thus, by applying theorem \ref{teo importante}, the  value function for the optimisation problem is:
\begin{equation}
\label{ex1}
V(x)= \begin{cases} 0 & x<\gamma\\
-\frac{20}{3}\left(15 - x\left(\frac{11}{3}+\frac{10}{\gamma}-\gamma\right) - \frac{3}{4} x^2-\frac{1}{x^2}\left(5\gamma^2-\frac{11}{3}\gamma^3+\frac{1}{4}\gamma^4\right) \right. & \\
\hspace{1.0cm} \left. +11x \left(\log(x) - \log(\gamma)\right)\right) & \gamma<x\leq x_0\\
\frac{20}{3}(\frac{1}{x^2}\left(\frac{1}{4}(\gamma^4-x_0^4)+\frac{11}{3}(x_0^3-\gamma^3)+5\gamma^2-7x_0^2-ex_0-2e\log |x_0-f| \right) &  \\
\hspace{1.0cm} +6-\frac{e}{2} +\frac{ex}{4}\log\left|\frac{x}{x-2}\right| +\frac{e}{x}+\frac{2e}{x^2}\log|x-2|) & x>x_0.
\end{cases}
\end{equation}
Note that a simple look at \eqref{ex1} reveals the following facts (frequently omitted from the analysis of optimal stopping time problem, notably in the framework of real options):
\begin{itemize}
\item The value function, the solution of an optimal stopping problem, may not be monotonic. However, in view of the verification theorem (see section \ref{secao teorema da verificacao}), it is always a smooth function;
\item This example clearly shows that the solution in the continuation region is quite different from the profit function $\Pi$, which suggests that, in general, it may be quite difficult to propose a particular solution. The representation of the value function $V$, here proposed, may be easier to use than the usual representation, where one needs to propose a particular solution of the ODE. 
\end{itemize}

In order to finish this subsection, we would like to illustrate the behaviour of the threshold $\gamma$ (as discussed in section \ref{Sensitivity Analysis}), for this particular example. As previously mentioned, when one considers a particular profit function $\Pi$, we need to carefully choose what domain of $r$, $\alpha$ and $\sigma^2$ is assumed. We start by noting that, in light of the condition $r>\sigma^2-\alpha$ presented above, we may have different cases. If the interest rate, $r$, is negative, then $\alpha>\frac{\sigma^2}{2}$ and, therefore, the decision of abandoning is anticipated with $\alpha$, but it is postponed with $\sigma^2$. When we have a null interest rate, one also must have $\alpha>\frac{\sigma^2}{2}$ and, consequently, the decision of exiting is not affected neither with $\alpha$ nor $\sigma^2$.   
Finally, when the interest rate is positive we have the usual results: $\gamma$ increases with $\alpha$ but the behaviour with $\sigma^2$ depends on the relation of $\alpha$ and $r$. Indeed, $\gamma$  decreases with $\sigma^2$ for $\alpha<r$ and increases when $\alpha>r$. Conversely, when $\alpha=r$, the decision of exiting remains unchanged. In table \ref{zero table examples}, we illustrate these different behaviours of the exit threshold with changing $\sigma^2$ for different scenarios of $r$ and $\sigma^2$.

\begin{table}[h!]
\centering
\caption{Thresholds' movement when $\sigma^2$ is increasing.}
\label{zero table examples}
\begin{tabular}{|c|c|c|c|c|}
\hline
$\sigma^2$                                                                  & $0.01$   & $0.05$   & $0.1$    & $0.15$   \\ \hline
\begin{tabular}[c]{@{}c@{}}$\gamma$ if\\ $r=0.1,~\alpha=0.09$\end{tabular}  & $0.3702$ & $0.3645$ & $0.3598$ & $0.3566$ \\ \hline
\begin{tabular}[c]{@{}c@{}}$\gamma$ if\\ $r=\alpha=0.1$\end{tabular}        & $0.3384$ & $0.3384$ & $0.3384$ & $0.3384$ \\ \hline
\begin{tabular}[c]{@{}c@{}}$\gamma$ if \\ $r=0.1,~\alpha=0.11$\end{tabular} & $0.3102$ & $0.3142$ & $0.3179$ & $0.3207$ \\ \hline
\end{tabular}
\end{table}
 
\subsection{Two sides stopping region : $K<0$}
In this subsection, we assume that: $a=1,\text{ }b=10,\text{ }c=1,\text{ }f=8$ and $K=-5$. We fix $r,\alpha$ and $\sigma^2$ as in the previous subsection, although the condition $r>\sigma^2-\alpha$ is not necessary in this example. In figure \ref{figure 2} we plot $\Pi$. The major difference regarding the previous case is that, now, $\Pi$ crosses level zero twice. Thus, we expect the continuation region to be two-sided (as this corresponds to case c) of definition \ref{def of x*}): $D^c=]\delta,\beta[$.

\begin{figure}[!h]\center
\caption{A gross profit function $\Pi$ when $K< 0$.}\label{figure 2}
\vspace{0.5cm}
\includegraphics[scale=0.7]{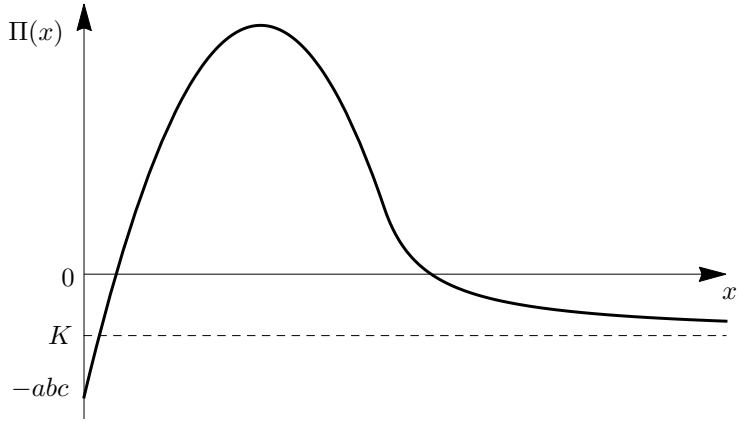}

\begin{picture}(160,10)

\put(-55,40){$K$}
\put(-70,155){$\Pi(x)$}
\put(200,57){$x$}
\put(-50,62){$0$}
\put(-69,20){$-abc$}
\end{picture}
\end{figure}
\vspace{-0.1cm}

According to \eqref{zeta, delta eq} in proposition \ref{x^* equation}, $(\delta,\beta)$ are solutions of the following integral equations:
\begin{align}\begin{cases}
&\int_{\delta}^{x_0}-cs^{-2}(s-1)(s-10)ds+\int_{x_0}^{\beta}s^{-2}\left(\frac{e}{x-8}-5\right)ds=0\\
&\int_{\delta}^{x_0}-cs(s-1)(s-10)ds+\int_{x_0}^{\beta}s\left(\frac{e}{x-8}-5\right) ds=0
\end{cases} \Rightarrow \begin{cases}
\delta\simeq 0.359353 \\
\beta \simeq 23.0984, \\
\end{cases}
\end{align}
where $x_0$ and $e$ are given as in  \eqref{x0}-\eqref{e}. 
Therefore, by applying theorem \ref{teo importante}, the  value function for the optimisation problem is:
\begin{equation}
\label{ex2}
V(x)= \begin{cases} 0 & x<\delta \text{ or } x>\beta\\
-\frac{20}{3}\left(15 - x\left(\frac{11}{3}+\frac{10}{\gamma}-\gamma\right) - \frac{3}{4} x^2-\frac{1}{x^2}\left(5\gamma^2-\frac{11}{3}\gamma^3+\frac{1}{4}\gamma^4\right)+ \right. & \\
\hspace{1.0cm} \left. +11x \left(\log(x) - \log(\gamma)\right)\right)  & \delta<x\leq x_0\\
-\frac{20}{3}(\frac{15}{2}+\frac{e}{8}+\log\left|\frac{\beta-8}{x-8}\right|\left(8ex^{-2}-\frac{e}{64}x\right)    +\frac{e}{64}x\log\left(\frac{\beta}{x}\right)+ \\
\hspace{1.0cm} +\frac{1}{x^2}\left(e\beta-\frac{5}{2}\beta^2\right)-ex^{-1} - x\left(5\beta^{-1}+\frac{e}{8}\beta^{-1}\right) &  x_0<x<\beta.
\end{cases}
\end{equation}

We finalise this section with a numerical illustration regarding the behaviour of the thresholds $\delta$ and $\beta$ with changing $\alpha$ and $\sigma^2$. Although we were not able to derive analytical results in the sensitivity analysis's section when $\Pi$ crosses level $0$ twice, the results here presented give us some clues about a more general case.

In table \ref{first table examples}, we represent the values of $\delta$ and $\beta$ for different values of $\alpha$, with positive interest rates, $r$. We remark that both $\delta$ and $\beta$ decrease with $\alpha$, as well as the amplitude of the continuation region. Given the shape of function $\Pi$ (the right-hand tail of $\Pi$ takes negatives values), this behaviour was expected, as higher $\alpha$ values also lead to a higher probability of having large values of $X$ and, consequently, bigger losses.

\begin{table}[h!]
\centering
\caption{Thresholds' movement when $\alpha$ is increasing, for $r=\sigma^2=0.1$}
\label{first table examples}
\begin{tabular}{|c|c|c|c|c|c|}
\hline
$\alpha$ & $-0.2$    & $-0.1$   & $0$      & $0.1$    & $0.2$    \\ \hline
$\delta$ & $0.819$   & $0.738$  & $0.571$  & $0.359$  & $0.220$  \\ \hline
$\beta$  & $143.199$ & $82.842$ & $42.218$ & $23.098$ & $16.649$ \\ \hline
\end{tabular}
\end{table}

A different behaviour is observed when the interest rate is negative, as one may see in table \ref{second table examples}. First of all, we notice that neither $\delta$ nor $\beta$ are monotonic with increasing $\alpha$, particularly when the pattern $\alpha<\frac{\sigma^2}{2}$ changes to $\alpha>\frac{\sigma^2}{2}$. Additionally, when $\alpha$ is between $0.2$ and $0.3$, the lower bound of the continuation region, $\delta$, is zero. This  was difficult to guess \textit{a priori}, since it means that it is never optimal to exit the project for small levels of the price. Finally, we note that, in this case, we have to exclude some values of $\alpha$ (notably, the ones between $-0.1$ and $0.2$) in order to meet condition \eqref{limite inferior par r}.

\begin{table}[h!]
\centering
\caption{Thresholds' movement when $\alpha$ is increasing, for $r=-0.1\text{ and }\sigma^2=0.1$}
\label{second table examples}
\begin{tabular}{|c|c|c|c|c|c|c|}
\hline
$\alpha$ & $-0.15$  & $-0.1$   & $0.2$    & $0.25$   & $0.3$    & $0.8$    \\ \hline
$\delta$ & $0.028$  & $0.013$  & $0$      & $0$      & $0$      & $0.002$  \\ \hline
$\beta$  & $12.139$ & $10.814$ & $23.093$ & $16.970$ & $15.090$ & $11.799$ \\ \hline
\end{tabular}
\end{table}

In the next two tables, we illustrate the behaviour of the thresholds $\delta$ and $\beta$, when $\sigma^2$ varies, for both positive and negative interest rates. In table \ref{third table examples}, we present results for the positive interest rate. We observe that both thresholds increase with $\sigma^2$, as well as the amplitude of the continuation region.

\begin{table}[h!]
\centering
\caption{Thresholds' movement when $\sigma^2$ is increasing, for $r=\alpha=0.1$}
\label{third table examples}
\begin{tabular}{|c|c|c|c|c|c|}
\hline
$\sigma^2$ & $0.1$    & $0.2$    & $0.3$    & $0.4$    & $0.5$     \\ \hline
$\delta$   & $0.359$  & $0.361$  & $0.3619$ & $0.3624$ & $0.3626$  \\ \hline
$\beta$    & $23.098$ & $41.985$ & $65.262$ & $91.678$ & $120.278$ \\ \hline
\end{tabular}
\end{table}

When the interest rate is negative (table \ref{fourth table examples}), for some levels of $\sigma^2$, $\delta=0$,  as in table \ref{second table examples}. Contrarily, when $\delta\neq 0$, $\delta$ increases with $\alpha$. Regarding $\beta$: it increases for \textit{low} values of $\sigma^2$ and then decreases. Although we were not able to prove it, it seems that $\beta$ increases with $\sigma^2$ while $\alpha>\frac{\sigma^2}{2}$, while decreasing otherwise. Please note the gap between $\sigma^2=0.1$ and $\sigma^2=1.9$. The reason for this gap is the same as for the gap observed in table \ref{second table examples}, and it is related with condition \eqref{limite inferior par r}.

\begin{table}[h!]
\centering
\caption{Thresholds' movement when $\sigma^2$ is increasing, for $r=-0.1\text{ and }\alpha=0.3$}
\label{fourth table examples}
\begin{tabular}{|c|c|c|c|c|c|c|c|}
\hline
$\sigma^2$ & $0.01$   & $0.05$   & $0.07$    & $0.1$    & $1.9$    & $2$      & $4$     \\ \hline
$\delta$   & $0$      & $0$      & $0$       & $0$      & $0.216$  & $0.227$  & $0.305$ \\ \hline
$\beta$    & $11.014$ & $12.307$ & $13.2431$ & $15.085$ & $770018$ & $143920$ & $4231$  \\ \hline
\end{tabular}
\end{table}
%
%
%
\appendix
\section{Auxiliary results to section \ref{Solution of the Optimal Stopping Time}}
\begin{prop}\label{auxiliar para condicoes integrais}
Let $\Pi:]0,\infty[\to]0,\infty[$ be a Borel measurable function, such that assumptions \ref{A3}-\ref{A1} hold true. Assuming that, in assumption \ref{A1}, $\Pi$ is such that
\begin{itemize}
\item[1)]$0<x_{1l}<x_{2r}=\infty$, then $\lim\limits_{x\to\infty}x^{d_1-d_2}\int_\nu^xs^{-d_1-1}\Pi(s)ds=0$;
\item[2)]$0=x_{1l}<x_{2r}<\infty$, then $\lim\limits_{x\downarrow 0}x^{d_2-d_1}\int_x^\nu s^{-d_2-1}\Pi(s)ds=0$.
\end{itemize}
\end{prop}
\begin{proof}
Since $v_p^+(x)<\infty$, a characterization for $\Pi^+$, such as in remark \ref{l1} is admissible. Furthermore, in case 1) $s^{-d_1-1}\Pi(s)=s^{d_2-d_1}s^{-d_2-1}\Pi(s)$ and $s^{-d_2-1}\Pi(s)=s^{d_1-d_2}s^{-d_1-1}\Pi(s)$. Thus, the result follows trivially.
\end{proof}

The following is a corollary of proposition \ref{auxiliar para condicoes integrais}, which gives an upper bound to $V$, given as in \eqref{value function}. Indeed, this result will be useful to prove theorem \ref{teo importante}.
\begin{corollary}\label{envelopes para v}
Suppose that assumptions \ref{A3}-\ref{A1} hold true, $V$ is given as in \eqref{value function} and $x^*$ is given by one of the conditions \eqref{gamma eq}-\eqref{zeta, delta eq}. Then, there are constants $b>0$ and $\beta_1,\beta_2\in\mathbb{R}$, such that if $\Pi$ verifies in assumption \ref{A1}:
\begin{itemize}
\item[1)] $0<x_{1l}<x_{2r}=\infty$, then $v(x)\leq b x^{\beta_2}$, for $\beta_2<d_2$; 
\item[2)] $0=x_{1l}<x_{2r}<\infty$, then  $v(x)\leq b x^{\beta_1}$, for $\beta_1>d_1$. 
\end{itemize}
\end{corollary}
\begin{proof}
Since the proofs of 1) and 2) require similar arguments, we will just prove 1). With trivial calculations, we obtain
\begin{align}\label{polinomio}
\lim\limits_{x\to\infty}\frac{v(x)}{x^{d_2}}=\frac{-2}{\sigma^2(d_2-d_1)}\lim_{x\to\infty}x^{d_1-d_2}\int_{x^*}^{x}s^{-d_1-1}\Pi(s)ds=0,
\end{align}
where the last equality  follows from proposition \ref{auxiliar para condicoes integrais}. The result follows, combining \eqref{polinomio} with the smoothness of the function $v$.
\end{proof}

\begin{prop}\label{HJB eq solution} 
Let $V\in C^1(]0,\infty[)$, with $V'\in AC(]0,\infty[)$, given in the same conditions of theorem \ref{teo importante}. Then, $V$ is a solution to the HJB equation \eqref{HJB}.
\end{prop}
\begin{proof} We start by proving the proposition assuming that $V$ is given as in point 1) of the theorem \ref{teo importante}. In order to do this, we prove that $V$ is non-negative, by contradiction. Therefore, we assume that $V(x)=\frac{-2}{\sigma^2(d_2-d_1)}\int_{x^*}^x\frac{\left(\frac{x}{s}\right)^{d_2}-\left(\frac{x}{s}\right)^{d_1}}{s}\Pi(s)ds<0$, for some $x\in D^c$, meaning that  $\int_{x^*}^x\frac{\left(\frac{x}{s}\right)^{d_2}-\left(\frac{x}{s}\right)^{d_1}}{s}\Pi(s)ds>0$, as $d_2-d_1>0$. As the solution to the ODE \eqref{ode} is continuous in the initial value $x^*$, then, there exists $\epsilon>0$, such that $\int_{x'}^x\frac{\left(\frac{x}{s}\right)^{d_2}-\left(\frac{x}{s}\right)^{d_1}}{s}\Pi(s)ds>0$, for all $x'\in B_{\epsilon}(x^*)$, where $B_{\epsilon}(x^*)$ is a ball of centre  $x^*$ and radius  $\epsilon$. Following definition \ref{def of x*}, when choosing $x'<\gamma$ in case a), $x'>\zeta$ in case b) or either $x'<\delta$ or $x'>\beta$  in case c), we get a contradiction. Thus,  $V(x)\geq 0$ for all  $x>0$. 

Next, we prove that, for $x \in  D$, $-{\cal L}v(x)-\Pi(x)\geq 0$, which is equivalent to prove that $\Pi(x)\leq 0$, when $x \in  D$. We start by assuming that we are in case a) of definition \ref{def of x*}. In that case, 
as $\frac{\left(\frac{x}{s}\right)^{d_2}-\left(\frac{x}{s}\right)^{d_1}}{s}>0$ for $x<s<x^*$, it follows that $\Pi(x^*)<0$. Otherwise, we would have a contradiction with the definition of $\gamma$ (see \eqref{gamma condition}). The cases b) and c) follow similarly, with the obvious changes.

If the function $V$ is defined as in the point 2) of theorem \ref{teo importante}, we note that $V$ admits the useful representation 
\begin{align}
V(x)&=\frac{2}{\sigma^2(d_2-d_1)}\left(x^{d_2}\int_{0}^{\infty}s^{-d_2-1}\Pi(s)ds-\int_{0}^x\frac{\left(\frac{x}{s}\right)^{d_2}-\left(\frac{x}{s}\right)^{d_1}}{s}\Pi(s)ds\right),\text{ for } x>0,
\end{align}
when the function $\Pi$ is such that $0<x_{1l}<x_{2r}=\infty$ in assumption \ref{A1} or
\begin{align} 
V(x)&=\frac{2}{\sigma^2(d_2-d_1)}\left(\int_{x}^\infty\frac{\left(\frac{x}{s}\right)^{d_2}-\left(\frac{x}{s}\right)^{d_1}}{s}\Pi(s)ds-x^{d_1}\int_{0}^{\infty}s^{-d_1-1}\Pi(s)ds\right),\text{ for } x>0,
\end{align}
when the function $\Pi$ is such that $0=x_{1l}<x_{2r}<\infty$ in assumption \ref{A1}. Combining these representations with definition \ref{def of x*} and proposition \ref{x^* equation}, we obtain $V(x)>0$ for all $x>0$. Then, the result follows from proposition 4.1 in Knudsen, Meister and Zervos \cite{knudsen1998valuation}. 
 
Finally, taking into account the comments about the regularity of the solution to the ODE \eqref{ode}, one must have $V\in C^1(]0,\infty[)$, with $V'\in AC(]0,\infty[)$.
\end{proof}

The following lemmas are related to the regularity condition \eqref{condi??o auxiliar}. This condition is crucial when one intends to prove that a function  $V\in C^1(]0,\infty[)$ with $V'\in AC(]0,\infty[)$, which is a solution to the HJB equation \eqref{HJB}, is, indeed, the value function to the optimal stopping problem \eqref{optimal00}.
\begin{lemma}\label{A.1}
Let  $X$ be a GBM satisfying \eqref{GBM} and consider $\beta\in\mathbb{R}$ and $a,b\in\mathbb{R}^+$. Then, the following equalities are true:
\begin{align*}
&E_x\left[e^{-rt}X_t^{\beta}{\cal I}_{\{X(t)>a\}}\right]=x^{\beta}e^{P(\beta)t}\left[1-H(t;a)\right];\\
&E_x\left[e^{-rt}X_t^{\beta}{\cal I}_{\{X(t)<b\}}\right]=x^{\beta}e^{P(\beta)t}H(t;b);\\
&E_x\left[e^{-rt}X_t^{\beta}{\cal I}_{\{a<X(t)<b\}}\right]=x^{\beta}e^{P(\beta)t}\left[H(t;b)-H(t;a)\right],
\end{align*}
where
\begin{equation}\label{auxiliar function}
H(t)=\Phi\left(\frac{1}{\sigma\sqrt{t}}ln\left(\frac{b}{x}\right)+\sigma\left(\frac{d_1+d_2}{2}-\beta\right)\sqrt{t}\right),
\end{equation} in which $\Phi$ is the cumulative density function of the standard normal distribution and $P$  is the characteristic polynomial defined in \eqref{characteristic polynomial}.
\end{lemma}
\begin{proof}
Let $A\subset]0,\infty[$ be an open set with strictly positive measure. Thus,
\begin{equation}\label{expectation}
E_x\left[e^{-rt}X^{\beta}(t){\cal I}_{\{X(t)\in A\}}\right]=x^{\beta}e^{P(\beta)t}\int_{\left\{\omega ~:~ xe^{-\frac{\sigma^2}{2}(d_1+d_2)t+\sigma \omega}\in A\right\}}\frac{e^{-\frac{(\omega-\beta\sigma t)^2}{2t}}}{\sqrt{2\pi t}}d\omega,
\end{equation}
the equality following in light of the parametrization, \eqref{reparametrization}, and the definition of the characteristic polynomial $P$, \eqref{polinomio}. By fixing  $A=]a,\infty[$, $A=]0,b[$ or $A=]a,b[$, and by using the change of variable $u=\frac{\omega-\sigma\beta t}{\sqrt{t}}$, we obtain the intended results.
\end{proof}

\begin{lemma}\label{A.2}
Let $X$ be a GBM satisfying \eqref{GBM} and consider $\beta\in\mathbb{R}$ and $a,b\in\mathbb{R}^+$.The following statements are true:
\begin{itemize}
\item[1)]If $\beta<d_2$, then $\lim\limits_{t\to\infty}x^{\beta}e^{P(\beta)t}\left[1-H(t;a)\right]=0$;
\item[2)]If $\beta>d_1$, then $\lim\limits_{t\to\infty}x^{\beta}e^{P(\beta)t}H(t;b)=0$;
\item[3)]If $\beta=0$, then $\lim\limits_{t\to\infty}=x^{\beta}e^{P(\beta)t}\left[H(t;b)-H(t;a)\right]=0$,
\end{itemize}
where $H$ is defined in lemma \ref{A.1}, and $P$ is the characteristic polynomial defined in \eqref{characteristic polynomial}. 
\end{lemma}
\begin{proof}
In order to prove the statements, we identify the following cases: (a) $\beta\in ]d_1,d_2[$, (b) $\beta\in\{d_1,d_2\}$ or (c) $\beta\in]0,\infty[\setminus[d_1,d_2]$. In situation (a), we have $P(\beta)<0$ and, consequently, the equalities are straightforward. If we are in situation (b), the equality in statement 1), 2) and 3) follows, respectively, in view of $\frac{d_1+d_2}{2}-d_1>0$,  $\frac{d_1+d_2}{2}-d_2<0$ and $\lim\limits_{t\to\infty}H(t;b)-H(t;a)=0$. Finally, in the case of (c), the result follows using the Cauchy rule.
\end{proof}

\section{Auxiliary results to section \ref{Sensitivity Analysis}}

In the following lemma, we present the behaviour of the roots of the characteristic equation, as given in \eqref{roots}. A similar result, but less general, can also be found in Guerra, Nunes and Oliveira \cite{nossoartigo}. 
\begin{lemma}\label{raizes decrescentes}
The functions $d_1(.,\sigma^2)$, $d_2(.,\sigma^2)$ and $d_2(\alpha,.)$ are strictly decreasing  and $d_1(\alpha,.)$ is a strictly increasing function. 
\end{lemma}
\begin{proof}
We start by calculating
$$
\frac{\partial d_1}{\partial \alpha}=\frac{1}{\sigma^2}\left(-1+\frac{\frac{\sigma^2}{2}-\alpha}{\sqrt{\left(\frac{\sigma^2}{2}-\alpha\right)^2+2\sigma^2r}}\right)\quad\text{and}\quad\frac{\partial d_2}{\partial \alpha}=\frac{1}{\sigma^2}\left(-1-\frac{\frac{\sigma^2}{2}-\alpha}{\sqrt{\left(\frac{\sigma^2}{2}-\alpha\right)^2+2\sigma^2r}}\right).
$$
Consequently, the described monotonicity of $d_1(.,\sigma^2)$ and $d_2(.,\sigma^2)$ follows straightforwardly. In order to study the monotonicity of $d_1(\alpha,.)$ and $d_2(\alpha,.)$, we analyse the respective derivatives:
\begin{align}
\frac{\partial d_i}{\partial \sigma^2}=\frac{2(-1)^i}{\sigma^4(d_2-d_1)}(\alpha d_i-r),\quad i=1,2.
\end{align}
The monotonicity of $d_1(\alpha,.)$ is straightforward when either $\left(r>0 \text{ and } \alpha\geq 0\right)$ or $\left(r=0\right)$ or $\left(r<0\text{ and }0\leq \alpha<\frac{\sigma^2}{2}\right)$. As for $d_2(\alpha,.)$, the same holds when either $\left(r>0 \text{ and } \alpha\leq 0\right)$ or $\left(r=0\right)$ or $\left(r<0\text{ and }0\leq \alpha<\frac{\sigma^2}{2}\right)$.  
Moreover, the signs of $\frac{\partial d_1}{\partial \alpha}$ and $\frac{\partial d_2}{\partial \alpha}$ may be obtained from, respectively, the signs of 
\begin{align}\label{d_1'sigma1}
\alpha\left(\sqrt{\left(\frac{\sigma^2}{2}-\alpha\right)^2+2\sigma^2r}-\left(\frac{\sigma^2}{2}-\alpha\right)\right)+r\sigma^2\quad \alpha\left(\sqrt{\left(\frac{\sigma^2}{2}-\alpha\right)^2+2\sigma^2r}+\left(\frac{\sigma^2}{2}-\alpha\right)\right)-r\sigma^2.
\end{align}
Taking this into account, it is a matter of calculations to obtain the results described in the tables \ref{table 1}-\ref{table 3}.
\end{proof}

\begin{lemma}\label{analise real}
Let $f:]0,\infty[\to\mathbb{R}$, such that $f(x)\leq 0$ for $x<a$, and $f(x)\geq 0$ for $x>a$, with $a>0$.  Moreover, assume that there exists  $z\in]0,\infty[$ (resp., $y\in]0,\infty[$) and $m\in \mathbb{R}$ (resp., $n\in \mathbb{R}$), such that
$\int_{z}^{\infty}s^{m}f(s)ds= 0$  (resp., $\int_{0}^{y}s^{n}f(s)ds= 0
$). 
Then,  
$$\int_{z}^{\infty}s^{m}\log(s)f(s)ds> 0 \quad \left(\text{resp., }\int_{0}^{y}s^{n}\log(s)f(s)ds> 0\right).$$
\end{lemma}
\begin{proof}
Considering that $\int_{z}^{\infty}s^{m}\log(s)f(s)ds=\int_{z}^{\infty}{s}^{m}\log\left(\frac{s}{z}\right)f(s)ds$, the result follows from
\begin{equation}\label{desigualdades comparativestatic 1}
\int_{z}^{\infty}{s}^{m}\log\left(\frac{s}{z}\right)f(s)ds>\log\left(\frac{a}{z}\right)\int_{z}^{a}{s}^{m}f(s)ds+\log\left(\frac{a}{z}\right)\int_{a}^{\infty}{s}^{m}f(s)ds=0.
\end{equation}
The second inequality in \eqref{desigualdades comparativestatic 1} may be proved using a similar argument.
\end{proof}

\section*{Acknowledgements}

Manuel Guerra was partially supported by the Project CEMAPRE - UID/MULTI/00491/2013 financed by FCT/MEC through national funds. Carlos Oliveira was supported by the Funda\c{c}\~ao para a Ci\^encia e Tecnologia (FCT) under Grant SFRH/BD/102186/2014.

\bibliographystyle{plain}
\bibliography{myrefs}

\begin{thebibliography}{10}

\bibitem{alvarez1999optimal}
Luis~HR Alvarez.
\newblock Optimal exit and valuation under demand uncertainty: A real options
  approach.
\newblock {\em European Journal of Operational Research}, 114(2):320--329,
  1999.

\bibitem{alvarez2006irreversible}
Luis~HR Alvarez and Erkki Koskela.
\newblock Irreversible investment under interest rate variability: Some
  generalizations.
\newblock {\em The Journal of Business}, 79(2):623--644, 2006.

\bibitem{arkin2015threshold}
VI~Arkin.
\newblock Threshold strategies in optimal stopping problem for one-dimensional
  diffusion processes.
\newblock {\em Theory of Probability \& Its Applications}, 59(2):311--319,
  2015.

\bibitem{belomestny2010optimal}
Denis Belomestny, Ludger R{\"u}schendorf, and Mikhail~Aleksandrovich Urusov.
\newblock Optimal stopping of integral functionals and a “no-loss” free
  boundary formulation.
\newblock {\em Theory of Probability \& Its Applications}, 54(1):14--28, 2010.

\bibitem{bronstein2006discretionary}
Anne~Laure Bronstein, Lane~P Hughston, Martijn~R Pistorius, and Mihail Zervos.
\newblock Discretionary stopping of one-dimensional ito diffusions with a
  staircase reward function.
\newblock {\em Journal of applied probability}, pages 984--996, 2006.

\bibitem{chevalier2015optimal}
Etienne Chevalier, Vathana~Ly Vath, Alexandre Roch, and Simone Scotti.
\newblock Optimal exit strategies for investment projects.
\newblock {\em Journal of Mathematical Analysis and Applications},
  425(2):666--694, 2015.

\bibitem{dahan2011impact}
Ely Dahan and V~Srinivasan.
\newblock The impact of unit cost reductions on gross profit: Increasing or
  decreasing returns?
\newblock {\em IIMB Management Review}, 23(3):131--139, 2011.

\bibitem{dayanik2008optimal}
Savas Dayanik.
\newblock Optimal stopping of linear diffusions with random discounting.
\newblock {\em Mathematics of Operations Research}, 33(3):645--661, 2008.

\bibitem{dayanik2012optimal}
Savas Dayanik and Masahiko Egami.
\newblock Optimal stopping problems for asset management.
\newblock {\em Advances in Applied Probability}, 44(03):655--677, 2012.

\bibitem{dayanik2003optimal}
Savas Dayanik and Ioannis Karatzas.
\newblock On the optimal stopping problem for one-dimensional diffusions.
\newblock {\em Stochastic Processes and their Applications}, 107(2):173--212,
  2003.

\bibitem{decamps2007optimal}
Jean-Paul D{\'e}camps and St{\'e}phane Villeneuve.
\newblock Optimal dividend policy and growth option.
\newblock {\em Finance and Stochastics}, 11(1):3--27, 2007.

\bibitem{dias2011hysteresis}
Jos{\'e}~Carlos Dias and Mark~B Shackleton.
\newblock Hysteresis effects under cir interest rates.
\newblock {\em European Journal of Operational Research}, 211(3):594--600,
  2011.

\bibitem{dixit1989entry}
A.~Dixit.
\newblock Entry and exit decisions under uncertainty.
\newblock {\em Journal of Political Economy}, pages 620--638, 1989.

\bibitem{DixitPindyck}
Avinash~K Dixit and Robert~S Pindyck.
\newblock {\em Investment under uncertainty}.
\newblock Princeton university press, 1994.

\bibitem{duckworth2000investment}
J~Kate Duckworth and Mihail Zervos.
\newblock An investment model with entry and exit decisions.
\newblock {\em Journal of applied probability}, 37(2):547--559, 2000.

\bibitem{filippov}
Aleksei~Fedorovich Filippov.
\newblock {\em Differential equations with discontinuous righthand sides:
  control systems}, volume~18.
\newblock Springer Science \& Business Media, 2013.

\bibitem{ghomrasni2004local}
Raouf Ghomrasni and Goran Peskir.
\newblock Local time-space calculus and extensions of it{\^o}’s formula.
\newblock In {\em High Dimensional Probability III}, pages 177--192. Springer,
  2004.

\bibitem{Griffin}
A.~Griffin and J.R. Hauser.
\newblock Integrating mechanisms for marketing and r\&d.
\newblock {\em Journal of Product Innovation Management}, 13(3):191--215, 1996.

\bibitem{nossoartigo}
Manuel Guerra, Cl{\'a}udia Nunes, and Carlos Oliveira.
\newblock Exit option for a class of profit functions.
\newblock {\em International Journal of Computer Mathematics}, pages 1--16,
  2016.

\bibitem{hagspiel2016escape}
Verena Hagspiel, Kuno~JM Huisman, Peter~M Kort, and Cl{\'a}udia Nunes.
\newblock How to escape a declining market: Capacity investment or exit?
\newblock {\em European Journal of Operational Research}, 254(1):40--50, 2016.

\bibitem{johnson2015solution}
Timothy~C Johnson.
\newblock The solution of some discretionary stopping problems.
\newblock {\em IMA Journal of Mathematical Control and Information}, page
  dnv060, 2015.

\bibitem{johnson2007solution}
Timothy~C Johnson and Mihail Zervos.
\newblock The solution to a second order linear ordinary differential equation
  with a non-homogeneous term that is a measure.
\newblock {\em Stochastics An International Journal of Probability and
  Stochastic Processes}, 79(3-4):363--382, 2007.

\bibitem{Kahn}
K.~B. Kahn.
\newblock Interdepartmental integration: a definition with implications for
  product development performance.
\newblock {\em Journal of Product Innovation Management}, 13(2):137--151, 1996.

\bibitem{knudsen1998valuation}
Thomas~S Knudsen, Bernhard Meister, and Mihail Zervos.
\newblock Valuation of investments in real assets with implications for the
  stock prices.
\newblock {\em SIAM journal on control and optimization}, 36(6):2082--2102,
  1998.

\bibitem{kobila1993}
T{\O}~Kobila.
\newblock A class of solvable stochastic investment problems involving singular
  controls.
\newblock {\em Stochastics and Stochastic Reports}, 43(1-2):29--63, 1993.

\bibitem{lamberton2013optimal}
Damien Lamberton, Mihail Zervos, et~al.
\newblock On the optimal stopping of a one-dimensional diffusion.
\newblock {\em Electron. J. Probab}, 18(34):1--49, 2013.

\bibitem{McDonald:Siegel:86}
R.~McDonald and D.~Siegel.
\newblock The value of waiting to invest.
\newblock {\em The Quarterly Journal of Economics}, 101(4):707--728, 1986.

\bibitem{Shiryaev}
Goran Peskir and Albert Shiryaev.
\newblock {\em Optimal stopping and free-boundary problems}.
\newblock Lectures in Mathematics ETH Z\"urich. Birkh\"auser Verlag, Basel,
  2006.

\bibitem{revuz2013continuous}
Daniel Revuz and Marc Yor.
\newblock {\em Continuous martingales and Brownian motion}, volume 293.
\newblock Springer Science \& Business Media, 2013.

\bibitem{ruschendorf2008class}
Ludger R{\"u}schendorf and Mikhail~A. Urusov.
\newblock On a class of optimal stopping problems for diffusions with
  discontinuous coefficients.
\newblock {\em The Annals of Applied Probability}, 18(3):847--878, 2008.

\bibitem{stokey2016wait}
Nancy~L Stokey.
\newblock Wait-and-see: Investment options under policy uncertainty.
\newblock {\em Review of Economic Dynamics}, 21:246--265, 2016.

\bibitem{trigeorgis}
Lenos Trigeorgis.
\newblock {\em Real options: Managerial flexibility and strategy in resource
  allocation}.
\newblock MIT press, 1996.

\bibitem{villeneuve2007threshold}
Stephane Villeneuve.
\newblock On threshold strategies and the smooth-fit principle for optimal
  stopping problems.
\newblock {\em Journal of Applied Probability}, pages 181--198, 2007.

\end{thebibliography}

\end{document}